\documentclass[11pt,reqno]{article}
\usepackage[final]{showkeys}
\usepackage[obeyspaces,hyphens,spaces]{url}

\renewcommand*\showkeyslabelformat[1]{%
\fbox{\parbox[t]{1.4 cm}{\raggedright\normalfont\small\url{#1}}}}

\RequirePackage{etex}
\usepackage{amsthm, amsmath, amsfonts, amssymb, mathtools, wasysym,  mathrsfs, empheq, etoolbox} 
\usepackage{mathpazo}

\usepackage[top= 2 cm, bottom = 2 cm, left = 2.2 cm, right= 2.2 cm]{geometry}
\usepackage[titletoc]{appendix}

\usepackage[table, dvipsnames]{xcolor} %

\usepackage[labelsep=space]{caption}  
\usepackage{subcaption}

\usepackage[doc]{optional}
\usepackage{graphicx, soul}

\usepackage{colortbl, float, booktabs, sectsty, multirow}
\restylefloat{table*}

\usepackage{pgf}
\usepackage{tikz,tikz-3dplot}
\usepackage{tkz-euclide}

\usetikzlibrary{arrows,quotes,angles,positioning,3d, intersections, calc,backgrounds}
\usetkzobj{all}

\renewcommand{\implies}{\Rightarrow}
\renewcommand{\iff}{\Leftrightarrow}

\colorlet{myblue}{blue}
\colorlet{mygreen}{Green}

\definecolor{myfirstblue}{rgb}{.8, .8, 1}

\newcommand*\mybluebox[1]{%
    \colorbox{RoyalBlue!20}{\hspace{1em}#1\hspace{1em}}}

\usepackage{hyperref}
\hypersetup{ 
	colorlinks= true, 
	urlcolor = Red, 
	citecolor = mygreen,
	linkcolor = myblue, 
	citebordercolor = myblue, 
	linkbordercolor = mygreen,
	urlbordercolor = Red
}

\allowdisplaybreaks

\numberwithin{equation}{section}

\usepackage[capitalize, nameinlink]{cleveref} 

\crefname{equation}{}{}

\crefname{chapter}{Appendix}{chapters}
\crefname{item}{}{items}
\crefname{figure}{Figure}{Figures}
\crefname{theorem}{\protect\theoremname}{\protect\theoremname}
\crefname{lemma}{\protect\lemmaname}{\protect\lemmaname}
\crefname{proposition}{\protect\propositionname}{\protect\propositionname}
\crefname{corollary}{\protect\corollaryname}{\protect\corollaryname}
\crefname{definition}{\protect\definitionname}{\protect\definitionname}
\crefname{fact}{\protect\factname}{\protect\factname}
\crefname{example}{\protect\examplename}{\protect\examplename}
\crefname{algorithm}{\protect\algorithmname}{\protect\algorithmname}
\crefname{remark}{\protect\remarkname}{\protect\remarkname}
\crefname{note}{\protect\notename}{\protect\notename}
\crefname{case}{\protect\casename}{\protect\casename}
\crefname{exercise}{\protect\exercisename}{\protect\exercisename}
\crefname{question}{\protect\questionname}{\protect\questionname}
\crefname{claim}{\protect\claimname}{\protect\claimname}
\crefname{enumi}{}{}

\usepackage{todonotes}

\makeatletter

\let\orgdescriptionlabel\descriptionlabel
\renewcommand*{\descriptionlabel}[1]{%
	\let\orglabel\label
	\let\label\@gobble
	\phantomsection
	\edef\@currentlabel{#1}%
	\let\label\orglabel
	\orgdescriptionlabel{#1}%
}

\let\leq\leqslant
\let\geq\geqslant

\def\th@plain{%
	\thm@notefont{} 
	\itshape 
}
\def\th@definition{%
	\thm@notefont{}
	\normalfont 
}

\g@addto@macro\th@remark{\thm@headpunct{}}
\g@addto@macro\th@definition{\thm@headpunct{}}
\g@addto@macro\th@plain{\thm@headpunct{}}

\makeatother

\theoremstyle{plain}

\newtheorem{theorem}{\protect\theoremname}[section]
\newtheorem{corollary}[theorem]{\protect\corollaryname}
\newtheorem{lemma}[theorem]{\protect\lemmaname}
\newtheorem{proposition}[theorem]{\protect\propositionname}
\newtheorem{fact}[theorem]{\protect\factname}

\theoremstyle{definition}

\newtheorem{remark}[theorem]{\protect\remarkname}

\newtheorem{example}[theorem]{\protect\examplename}

\newtheorem*{notation}{\protect\notationname}

\theoremstyle{remark}

\providecommand{\theoremname}{Theorem}
\providecommand{\propositionname}{Proposition}
\providecommand{\corollaryname}{Corollary}
\providecommand{\factname}{Fact}
\providecommand{\lemmaname}{Lemma}

\providecommand{\definitionname}{Definition}
\providecommand{\notationname}{Notation}
\providecommand{\remarkname}{Remark}

\providecommand{\examplename}{Example}
\providecommand{\observationname}{Observation}
\providecommand{\claimname}{Claim}

\providecommand{\problemname}{Problem}  
\providecommand{\algorithmname}{Algorithm}
\providecommand{\exercisename}{Exercise}
\providecommand{\casename}{Case}
\providecommand{\questionname}{Question}
\providecommand{\notename}{Note}

\renewcommand\theenumi{(\roman{enumi})}

\renewcommand{\labelenumi}{\rm (\roman{enumi})}

\newcommand\T{%
	{\mathchoice
		{\raisebox{.5ex}{$\displaystyle{\intercal}$}}
		{\raisebox{.5ex}{$\textstyle{\intercal}$}}
		{\raisebox{.35ex}{$\scriptstyle{\intercal}$}}  
		{\raisebox{.35ex}{$\scriptscriptstyle{\intercal}$}}}
}

\newcommand{\ball}[2]{\ensuremath{  \operatorname{B} \left({#1};{#2}\right) } }
\newcommand{\sphere}[2]{\ensuremath{ \operatorname{S}   \left({#1};{#2}\right)  } }

\newcommand{\pc}[1]{{#1}^{\ominus}}

\newcommand{\dsphere}[2]{\ensuremath{ \boldsymbol{\operatorname{S}}   \left({#1};{#2}\right)  } }

\DeclarePairedDelimiterX\menge[2]{ \{ }{ \} }{ {#1} ~ \delimsize \vert ~ \mathopen{} {#2} }
\DeclarePairedDelimiterX\fa[2]{ ( }{ )_{#2} }{#1}
\DeclarePairedDelimiterX\set[2]{ \{ }{ \}_{#2} }{#1}

\DeclarePairedDelimiterX\rb[1]{ ( }{ ) }{#1}

\DeclarePairedDelimiterX\scal[2]{\langle}{\rangle}%
{ \ifblank{#1#2}{ \, \cdot \, \delimsize \vert \, \mathopen{}\cdot \, } {   {#1} \, \delimsize \vert \, \mathopen{}{#2}  } }

\DeclarePairedDelimiterX\norm[1]{\lVert}{\rVert}%
{ \ifblank{#1}{\, \cdot \,}{#1} }

\DeclarePairedDelimiterXPP\fnorm[1]{}\lVert\rVert{_\ensuremath{\mathsf{F}}}{ \ifblank{#1}{\, \cdot \,}{#1} }

\DeclarePairedDelimiterXPP\fscal[2]{}\langle\rangle{_\ensuremath{\mathsf{F}}}{ \ifblank{#1#2}{ \, \cdot \, \delimsize \vert \,\mathopen{} \cdot \, }{  {#1} \, \delimsize \vert \, \mathopen{} {#2}  } }

\DeclarePairedDelimiterX\abs[1]{\lvert}{\rvert}{\ifblank{#1}{\, \cdot \,}{#1}}

\newcommand{\minimize}[2]{\ensuremath{\underset{\substack{{#1}}}{\mathrm{minimize}}~~{#2} }}

\newcommand{\HH}{\ensuremath{\mathcal H}}

\newcommand{\RR}{\ensuremath{\mathbb R}}
\newcommand{\NN}{\ensuremath{\mathbb N}}

\newcommand{\bbS}{\ensuremath{\mathbb S}}

\newcommand{\bbU}{\ensuremath{\mathbb U }}

\newcommand{\sC}{\ensuremath{\mathsf C}}
\newcommand{\sD}{\ensuremath{\mathsf D}}
\newcommand{\sE}{\ensuremath{\mathsf E}}

\newcommand{\bx}{\ensuremath{ \boldsymbol{x} } }
\newcommand{\by}{\ensuremath{\boldsymbol{y} } }

\newcommand{\bC}{\ensuremath{\boldsymbol{C}  }}

\newcommand{\bK}{\ensuremath{{\boldsymbol{K} }}}

\newcommand{\bzero}{\ensuremath{{\boldsymbol{0}}}}

\newcommand{\Diag}{\ensuremath{\operatorname{Diag} }}
\DeclareMathOperator{\tra}{tra}
\DeclareMathOperator{\rank}{rank}

\newcommand{\closu}[1]{\ensuremath{\overline{#1} }}

\newcommand{\conv}{\ensuremath{\operatorname{conv}}}
\newcommand{\cone}{\ensuremath{\operatorname{cone}}}
\newcommand{\ccone}{\ensuremath{\operatorname{\overline{cone}}}}

\newcommand{\Id}{\ensuremath{\operatorname{Id}}}

\newcommand{\rec}{\ensuremath{\operatorname{rec}}}

\newcommand{\pos}{ \ensuremath{\operatorname{pos}} }

\let\originalleft\left

\renewcommand{\left}{\mathopen{}\originalleft}

\newcolumntype{M}[1]{>{\centering\arraybackslash$}m{#1}<{$}   }

\let\b\boldsymbol 

\begin{document}

\title{ \sffamily  Projecting onto the intersection of a cone and
a sphere 
 }

\author{
  	Heinz H.\ Bauschke\thanks{
  		Mathematics, University of British Columbia, Kelowna, B.C.\ V1V~1V7, Canada. 
  		Email: \href{mailto: heinz.bauschke@ubc.ca}{\texttt{heinz.bauschke@ubc.ca}}.},~
  	Minh N. Bui\thanks{Mathematics, University of British Columbia, Kelowna, B.C.\ V1V~1V7, Canada. Email: \href{mailto: nhutminh.bui@alumni.ubc.ca}{\texttt{nhutminh.bui@alumni.ubc.ca}}.},~
  	and Xianfu Wang\thanks{
  	Mathematics, University of British Columbia, Kelowna, B.C.\ V1V~1V7, Canada. 
  	Email: \href{mailto: shawn.wang@ubc.ca}{\texttt{shawn.wang@ubc.ca}}.}
}

\date{April 12, 2018}

\maketitle

\begin{abstract}
\noindent
The projection onto the intersection of sets generally does not
allow for a closed form even when the individual projection
operators
have explicit descriptions. 
In this work, we systematically analyze the projection onto the
intersection of a cone with either a ball or a sphere.
Several cases are provided where the projector is available in
closed form. Various examples based on finitely generated cones,
the Lorentz cone, and the cone of positive semidefinite
matrices are presented. The usefulness of our formulae is 
illustrated by numerical experiments for determining copositivity
of real symmetric matrices. 
\end{abstract}
{\small
\noindent
{\bfseries 2010 Mathematics Subject Classification:}
{Primary 
47H09, 
52A05, 
90C25, 
90C26;
Secondary 
15B48, 
47H04. 
}

\noindent {\bfseries Keywords:}
ball,
convex set,
convex cone,
copositive matrix,
projection,
projector,
sphere.
}

\section{Introduction}

Throughout this paper, we assume that 
	\begin{empheq}[box = \mybluebox]{equation}\label{H}
		 \text{$\HH$ is a real Hilbert space} 
	\end{empheq}
with inner product $\scal{\cdot}{\cdot}$ and induced norm
$\|\cdot\|$. 
Let $K$ and $S$ be subsets of $\HH$, with
associated projection operator (or projectors) 
\begin{equation}
\text{$P_K$ and $P_S$,}
\end{equation}
respectively. Our aim is to derive a formula for the projector of
the intersection 
\begin{equation}
\text{$P_{K\cap S}$.}
\end{equation}
Only in rare cases  is it possible
to obtain a {\textquotedblleft}closed form{\textquotedblright} for $P_{K\cap S}$ in terms of $P_K$
and $P_S$: e.g., when $K$ and $S$ are either both half-spaces
(Haugazeau; see \cite{haugazeau1968inequations} and also \cite[Corollary~29.25]{bauschke2017convex})
or both subspaces (Anderson{\textendash}Duffin; see \cite[Theorem~8]{anderson1969series} and also
\cite[Corollary~25.38]{bauschke2017convex}). 
Inspired by an example in the recent and charming book
\cite{lange2016mm}, \emph{our aim in this paper is to systematically
study the case when $K$ is a closed convex \emph{cone} and $S$ is
either the (convex) unit ball or (nonconvex) unit sphere centered
at the origin.} In \cite[Example~5.5.2]{lange2016mm}, 
Lange used this projector
for an algorithm on determining copositivity of a matrix;
however, this projection has the potential to be useful in
other settings where, say, \emph{a priori} constraints are present
(e.g., positivity and energy). We obtain formulae describing the
full (possibly set-valued) projector and also discuss
nonpolyhedral cones such as the Lorentz cone or the cone of
positive semidefinite matrices. We also revisit Lange's
copositivity example and tackle it with other algorithms that
appear to perform quite well. 

The remainder of the paper is organized as follows.
 \cref{s:aux} contains miscellaneous results for subsequent 
use. In \cref{s:cones}, we provide various results on
cones and conical hulls.
The description of projections involving cones and subsets of
spheres is the topic of \cref{s:proj}. 
In \cref{s:Sn}, we turn to results formulated in 
the Hilbert space of real symmetric matrices. 
Cones that are finitely generated and corresponding projectors
are investigated in \cref{s:fg}. 
Our main results are presented in \cref{s:main1} (cone
intersected with ball) and \cref{s:main2} (cone
intersected with sphere), respectively.
Additional examples are provided in \cref{s:fex}. 
In the final \cref{s:copos}, we put the theory to good use
and offer new algorithmic approaches to determine copositivity. 

We conclude this introductory section with some comments on
notation. 
For a subset $C$ of $\HH$, its  \emph{closure} (with
respect to the norm topology of $\HH$) and \emph{orthogonal
complement}
are denoted by  $\closu{C}$ and $C^{\perp}$, respectively. 
Next, $\NN \coloneqq \left\{0,1,2,\ldots
\right\}$,  $\RR_{+} \coloneqq \left[ 0, +\infty \right[$, and
$\RR_{++} \coloneqq \left] 0,+\infty\right[$. In turn, the sphere
and the closed ball in $\HH$ with center $x \in \HH$ and radius
$\rho \in \RR_{++}$ are respectively defined as $\sphere{x}{\rho}
\coloneqq \menge{y \in \HH}{\norm{y-x} = \rho}$ and $\ball{x}{\rho}
\coloneqq \menge{y \in \HH}{ \norm{y-x} \leq \rho}$. 
The product space 
	$\b{\HH} \coloneqq \HH\oplus \RR$ is 
equipped with the  scalar product $  \rb{\rb{x,\xi}, \rb{y,\eta}}
\mapsto \scal{x}{y} + \xi\eta$, and we shall use boldface letters
for sets and vectors in $\b{\HH}$. The notation 
mainly follows \cite{bauschke2017convex} or 
will be introduced as needed. 

\section{Auxiliary results}

\label{s:aux}

In this short section, we collect a few results that will  be
useful later. 

\begin{lemma}\label{lem: numb0}
	Let $\set{\alpha_{i}}{i \in I}$ be a finite subset of $\RR$ such that 	
		\begin{equation}\label{eq: num1}
			\left(\forall i \in I\right)\left(\forall j \in I\right) \quad i \neq j \implies \alpha_{i}\alpha_{j}=0
		\end{equation}
	and that 
		\begin{equation}\label{eq: num2}
			\sum_{i \in I}\alpha_{i} = 1.
		\end{equation}
	Then there exists $i \in I$ such that $\alpha_{i} =1$ and $\left(\forall j \in I \smallsetminus \left\{i\right\}\right)~ \alpha_{j}=0$.
\end{lemma}

	\begin{proof}
		Suppose that there exist $i$ and $j$ in $I$ such that $i \neq j$, that $\alpha_{i} \neq 0$, and that $\alpha_{j} \neq 0$. Then $\alpha_{i}\alpha_{j} \neq 0$, which violates \cref{eq: num1}. Hence, $\fa{\alpha_{i}}{i \in I}$ contains at most one nonzero number. On the other hand, by \cref{eq: num2}, $\left(\alpha_{i}\right)_{i \in I}$ must contain at least one nonzero number. Altogether, we conclude that there exists $i \in I$ such that $\alpha_{i} \neq 0$ and $\left(\forall j \in I\smallsetminus\left\{i\right\}\right)~\alpha_{j} =0$. Consequently, it follows from \cref{eq: num2} that  $\alpha_{i} =1$, as claimed.
	\end{proof}

\begin{lemma}\label{lem: numb}
	Let $\left\{x_{i}\right\}_{i \in I}$ be a finite subset
	of $\HH$, and let $\left\{ \alpha_{i} \right\}_{i \in I}$ be a 
	finite subset of $\RR$ such that $\sum_{i \in I}\alpha_{i} =1$. Set $x \coloneqq \sum_{i \in I}\alpha_{i} x_{i}$ and $\beta \coloneqq \norm{x}$. Then the following hold: 
		\begin{enumerate}
			\item\label{it: vt1} $ \beta^{2} + \sum_{\left(i,j\right) \in I \times I } \alpha_{i} \alpha_{j} \norm{x_{i}-x_{j}}^{2}/2 = \sum_{i \in I} \alpha_{i} \norm{x_{i}}^{2}$.
			\item\label{it: vt2} Suppose that 
				\begin{equation}\label{eq: vt-as1}
					\left(\forall i \in I\right) \quad \norm{x_{i}} =\beta,
				\end{equation}
			that 
				\begin{equation}\label{eq: vt-as2}
					\left(\forall i \in I\right) \quad \alpha_{i} \geq 0, 
				\end{equation}
			and that the vectors $\set{x_{i}}{i \in I}$ are pairwise distinct, i.e., 	
				\begin{equation}\label{eq: vt-as3}
					\left(\forall i \in I\right)\rb{\forall j \in I} \quad i \neq j \implies x_{i} \neq x_{j}.
				\end{equation}
			Then $\left(\exists i \in I\right)~ x= x_{i}$.
		\end{enumerate}
\end{lemma}

	\begin{proof}
		\cref{it: vt1}:  See, for instance, \cite[Lemma 2.14(ii)]{bauschke2017convex}.
		
		\cref{it: vt2}: Since $\sum_{i \in I} \alpha_{i} =1$, we deduce from \cref{it: vt1} and \cref{eq: vt-as1} that $ \beta^{2} + \sum_{\left(i,j\right) \in I \times I } \alpha_{i} \alpha_{j} \norm{x_{i}-x_{j}}^{2}/2 = \sum_{i \in I} \alpha_{i} \beta^{2} = \beta^{2} $, which yields $\sum_{\left(i,j\right) \in I \times I } \alpha_{i} \alpha_{j} \norm{x_{i}-x_{j}}^{2} =0$ or, equivalently, by \cref{eq: vt-as2}, 
			\begin{equation}\label{eq: vt-1}
				\left(\forall i \in I\right)\left(\forall j \in I\right) \quad \alpha_{i} \alpha_{j} \norm{x_{i} - x_{j}}^{2} = 0. 
			\end{equation}
		Thus,  we get from \cref{eq: vt-as3} and \cref{eq: vt-1} that 
			\begin{math}
				\left(\forall i \in I\right)\left(\forall j \in I\right)~ i \neq j \Rightarrow \norm{x_{i} - x_{j}} \neq 0 \Rightarrow \alpha_{i}\alpha_{j} =0,
			\end{math}
		and because $\sum_{i \in I}\alpha_{i} =1$, \cref{lem: numb0} guarantees the existence of $i \in I$ such that $\alpha_{i} =1$ and $\left(\forall j \in I \smallsetminus \left\{i\right\}\right)~\alpha_{j} =0$. Consequently, it follows from the very definition of $x$ that $x=x_{i}$, as desired.
	\end{proof}

\begin{lemma}\label{lem: maxSph}
	Let $\alpha$  be in $\RR$, let $\beta$ be in $\RR_{++}$,
	and let $\bx = \rb{x,\xi} \in \b{\HH}$.
	Set\footnote{Here and elsewhere,
	``$\times$'' denotes the \emph{Cartesian product} of
	sets.}	
		\begin{equation}
			\b{S}_{\alpha,\beta} \coloneqq \sphere{0}{\beta} \times \left\{ \alpha \right\}.
		\end{equation}
	Then $\max \scal{\bx}{\b{S}_{\alpha,\beta} } = \beta \norm{x} + \xi \alpha$.
\end{lemma}

	\begin{proof}
		We shall assume that $x \neq 0$, since otherwise $\scal{\bx}{\b{S}_{\alpha,\beta} } = \{ \xi \alpha \}$ and the assertion is clear. Now, for every $\b{y}=\rb{y,\alpha} \in \b{S}_{\alpha,\beta}$, since $\norm{y} =\beta$,  the Cauchy{\textendash}Schwarz inequality yields 
			\begin{equation}
				\scal{\bx}{\b{y} }= \scal{x}{y} + \xi\alpha \leq \norm{x}\norm{y} + \xi \alpha = \beta \norm{x} + \xi \alpha.
			\end{equation}
		Hence $\sup \scal{\bx}{\b{S}_{\alpha,\beta} } \leq  \beta \norm{x} + \xi \alpha$. Consequently, because $\rb{\beta x /\norm{x} , \alpha } \in \b{S}_{\alpha,\beta}$ and  
			\begin{equation}
				\scal*{\bx}{\rb*{\frac{\beta x}{\norm{x}} , \alpha } } = \scal*{x}{ \frac{\beta x}{\norm{x}} } + \xi \alpha = \beta\norm{x} + \xi \alpha, 
			\end{equation}
		we obtain the conclusion.
	\end{proof}

\section{Cones and conical hulls}
\label{s:cones}

In this section, we systematically study cones and conical hulls.

Let $C$ be a subset of $\HH$. Recall that the 
\emph{convex hull} of $C$, i.e., the smallest convex subset of
$\HH$ containing $C$, is denoted by $\conv{C}$ and 
(see, e.g., \cite[Proposition 3.4]{bauschke2017convex}), it is characterized by
	\begin{equation}\label{eq: conv}
		\conv{C} = \menge*{ \sum_{i\in I}  \alpha_{i} x_{i} }{ \text{$I$ is finite, $\left\{\alpha_{i} \right\}_{i\in I} \subseteq \left]0,1 \right]$ such that $\sum_{i\in I}\alpha_{i}=1$, and $\left\{x_{i}\right\}_{i\in I} \subseteq C$ }  }.
	\end{equation}
Next, $C$ is a \emph{cone} if $C = \bigcup_{\lambda \in \RR_{++}} \lambda C$. In turn, the \emph{conical hull} of $C$ is the smallest cone in $\HH$ containing $C$ and is denoted by $\cone{C};$ furthermore, the \emph{closed conical hull} of $C$, in symbol, $\ccone{C}$, is the smallest closed cone in $\HH$ containing $C$. Finally, the \emph{polar cone} of $C$ is 
	\begin{equation}
		\pc{C} \coloneqq \menge*{u \in \HH}{\sup \scal{u}{C} \leq 0 },
	\end{equation}
and the \emph{recession cone} of $C$ is 
	\begin{equation}
		\rec{C} \coloneqq \menge*{x \in \HH}{x + C \subseteq C}.
	\end{equation}

\begin{example}\label{eg: co-S}
	Let $\rho \in \RR_{++}$, and set $C \coloneqq \sphere{0}{\rho}$. Then $\conv{C} = \ball{0}{\rho}$.
\end{example}

	\begin{proof}
		Since $\ball{0}{\rho}$ is convex and $C \subseteq
		\ball{0}{\rho}$, we obtain $\conv{C} \subseteq
		\ball{0}{\rho}$. Conversely, take $x \in
		\ball{0}{\rho}$, and we consider the following 
		two conceivable cases: 
			
			\hspace{\parindent}\emph{Case 1:} $x=0$: Fix $y \in C$. Then clearly $-y \in C$ and $x = 0 = \rb{1/2}y + \rb{1/2}\rb{-y} \in \conv{C}$.
			
			\hspace{\parindent}\emph{Case 2:} $x \neq 0$: Set $x_{+}\coloneqq \rb{\rho /\norm{x}}x$, $x_{-}\coloneqq \left(\rho/\norm{x}\right)\rb{-x}$, and $\alpha \coloneqq \rb{1+\norm{x}/\rho} /2 $. Then $\{ x_{+},x_{-} \} \subseteq C$, and because $\norm{x} \leq \rho$, we have $\alpha \in \left]0,1\right]$. Thus, since it is readily verified that $x = \alpha x_{+} + \rb{1-\alpha}x_{-}$, we get $x \in \conv{C}$. 
			
		Hence, $x \in \conv{C}$ in both cases, which completes the proof.
	\end{proof}

For the sake of clarity, let us point out the following.

\begin{remark}\label{rem: clar}
	Let $K$ be a nonempty cone in $\HH$. Then $0 \in \closu{K}$, and if $K \neq \left\{0\right\}$, then $\left(\forall \rho \in \RR_{++}\right)~ K \cap \sphere{0}{\rho} \neq \varnothing$.
\end{remark}

\begin{fact}\label{fact: cone}
	Let $C$ be a subset of $\HH$. Then the following hold: 
	\begin{enumerate}
		\item\label{it: cone1} $\cone{C} = \bigcup_{\lambda \in \RR_{++}} \lambda C$.
		\item\label{it: cone2} $\ccone{C}  = \closu{ \cone{C} }$.
		\item\label{it: cone3} $\cone \left(\conv{C}\right)$ is the smallest convex cone containing $C$.
	\end{enumerate}
\end{fact}

\begin{proof}
	See, e.g., \cite[Proposition 6.2(i){\textendash}(iii)]{bauschke2017convex}.
\end{proof}


In general, for subsets $C$ and $D$ of $\HH$, $\closu{C \cap D} \neq \closu{C} \cap \closu{D}$. 
However, the following result provides an interesting instance
where taking intersections and closures commutes.

\begin{proposition}\label{lem: topo}
	Let $K$ be a nonempty cone in $\HH$, and let $\rho \in \RR_{++}$. Then the following hold: 
	\begin{enumerate}
		\item\label{it: topo1}	$\closu{K} \cap \sphere{0}{\rho} = \closu{ K \cap \sphere{0}{\rho}}$.
		\item\label{it: topo2} $\closu{K} \cap \ball{0}{\rho} = \closu{ K \cap \ball{0}{\rho}}$.
	\end{enumerate}
\end{proposition}

\begin{proof}
	We assume that  
		\begin{equation}\label{eq: topo-zero}
			K \ne \left\{0\right\},
		\end{equation}
	since otherwise the assertions are clear.
	
	\cref{it: topo1}: Since we obviously have $\closu{ K \cap \sphere{0}{\rho}} \subseteq \closu{K} \cap \sphere{0}{\rho}$, it suffices to verify that $\closu{K} \cap \sphere{0}{\rho} \subseteq \closu{ K \cap \sphere{0}{\rho}}$. To do so, take $x \in \closu{K} \cap \sphere{0}{\rho}$. Then, because $x \in \closu{K}$, there exists a sequence $\left(x_{n}\right)_{ n \in \NN}$ in $K$ such that 
		\begin{equation}\label{eq: topo1}
			x_{n} \to x.
		\end{equation} 
	In turn, by the continuity of $\norm{}$ and the fact that $x \in \sphere{0}{\rho}$, 
	\begin{equation}\label{eq: topo2}
	\norm{x_{n}} \to \norm{x} = \rho \in \RR_{++},
	\end{equation}
	and therefore, we can assume without loss of generality that $\left(\forall n \in \NN\right)~ \norm{x_{n}} \neq 0$. Hence, for every $n \in \NN$, since $x_{n} \in K$ and $\norm{ \rho x_{n} / \norm{x_{n}} } = \rho$, the assumption that $K$ is a cone implies that $\rho x_{n} / \norm{x_{n}}$ lies in $ K \cap \sphere{0}{\rho}$. Thus, $\left(\rho x_{n}/\norm{x_{n}}\right)_{n \in \NN}$ is a sequence in $ K \cap \sphere{0}{\rho}$; moreover, \cref{eq: topo1} and \cref{eq: topo2} assert that $\rho x_{n}/ \norm{x_{n}} \to \rho x / \rho =x$. Consequently, $x \in \closu{K \cap \sphere{0}{\rho}}$, as announced.
	
	\cref{it: topo2}: First, it is clear that $ \closu{ K \cap \ball{0}{\rho}} \subseteq \closu{K} \cap \ball{0}{\rho}$. Conversely, fix $x \in \closu{K} \cap \ball{0}{\rho}$, and we shall consider two conceivable cases: 
	
	\hspace{\parindent}(A) $x=0$: By \cref{eq: topo-zero}, there exists  
	\begin{equation}\label{eq: topo3}
	y \in K \smallsetminus \left\{0\right\}. 
	\end{equation}
	In turn, set 
	\begin{equation}\label{eq: topo4}
	\left(\forall n \in \NN\right) \quad y_{n} \coloneqq \frac{\rho}{\left(n+1\right)\norm{y}}y. 
	\end{equation}
	Then, for every $n \in \NN$, since $K$ is a cone, \cref{eq: topo3} and \cref{eq: topo4} assert that $y_{n} \in K$ and thus, since $\norm{y_{n} } = \rho / \left(n+1\right) \leq \rho$ by \cref{eq: topo4}, we deduce that $y_{n} \in K \cap \ball{0}{\rho}$. Hence, because $y_{n} = \rho y / \left[ \left( n+1 \right) \norm{y}\right] \to 0 =x $, we infer that $x \in \closu{K \cap \ball{0}{\rho}}$. 
	
	\hspace{\parindent}(B) $x \neq 0$: Since $x \in \closu{K}$, there is a sequence $\left(x_{n}\right)_{n \in \NN}$ in $K$ such that 
		\begin{equation}\label{eq: topo5}
			x_{n} \to x.
		\end{equation}
	In turn, by the continuity of $\norm{}$, 
		\begin{equation}\label{eq: topo6}
			\norm{x_{n}} \to \norm{x} \in \RR_{++},
		\end{equation}
	and we can therefore assume that $\left(\forall n \in \NN\right)~ \norm{x_{n}} \neq 0$. Now set 
		\begin{equation}\label{eq: topo7}
			\left(\forall n \in \NN\right) \quad y_{n} \coloneqq \frac{\norm{x}}{\norm{x_{n}}}x_{n}.
		\end{equation}
	For every $n \in \NN$, because $x_{n} \in K$ and $\norm{x}  \leq \rho$,    the assumption that $K$ is a cone and \cref{eq: topo7} yield $y_{n} \in K \cap \ball{0}{\rho}$. Consequently, since  $y_{n} \to \norm{x} x / \norm{x} = x$ due to \cref{eq: topo5} and \cref{eq: topo6}, we obtain $x \in \closu{K \cap \ball{0}{\rho}}$.
	
	To sum up, in both cases, we have $x \in \closu{K \cap \ball{0}{\rho}}$, and the conclusion follows.
\end{proof}


We shall require the following notation.

\begin{notation}
	Let $C$ be a nonempty subset of $\HH$. Define 
	its \emph{positive
	span}\footnote{Technically, this should
	be called 
	the ``nonnegative span'' but we follow the more common
	usage.} by
		\begin{empheq}[box = \mybluebox]{equation}\label{eq: Rc}
		 	\pos{C} \coloneqq \menge*{ \sum_{i \in I} \alpha_{i} x_{i} }{\text{$I$ is finite, $\left\{\alpha_{i}\right\}_{i \in I} \subseteq \RR_{+}$, and $\left\{x_{i}\right\}_{i \in I} \subseteq C$}}.
		\end{empheq}
	We observe that if $C$ is finite, then $\pos{C}$ coincides\footnote{This readily follows from \cref{eq: Rc}.} with the Minkowski sum of the sets $\left(\RR_{+}c\right)_{c \in C}$, i.e., 
		\begin{equation}\label{eq: Minsum}
			\pos{C} = \sum_{c \in C} \RR_{+}c.
		\end{equation}
\end{notation}

\begin{lemma}\label{lem: cone}
	Let $C$ be a nonempty subset of $\HH$, and set 	
		\begin{math}
			K \coloneqq \pos{C}.
		\end{math}
	Then the following hold: 
	\begin{enumerate}
		\item\label{it: c01} $x \in \pc{K} \Leftrightarrow \sup \scal{x}{C} \leq 0$.
		\item\label{it: c1} \begin{math}
		K = \cone  \left( {\conv}{\left(C \cup \left\{0\right\}\right) } \right)  = \cone   \left(\left\{0\right\} \cup \conv{C} \right) = \cone  \left(\conv{C}\right) \cup \left\{0\right\}.
		\end{math}
		\item\label{it: c2} $K$ is the smallest convex cone containing $C\cup\left\{0\right\}$.
	\end{enumerate}
\end{lemma}

\begin{proof}
	\cref{it: c01}: This follows from the definition of polar cones and \cref{eq: Rc}.
	
	\cref{it: c1}: Set $D \coloneqq \cone  \left( {\conv}{\left(C \cup \left\{0\right\}\right) } \right)$, $E \coloneqq \cone   \left(\left\{0\right\} \cup \conv{C} \right)$, and $F \coloneqq \cone  \left(\conv{C}\right) \cup \left\{0\right\}$. We shall establish that 
		\begin{equation}\label{eq: co-todo}
			K \subseteq D \subseteq E = F \subseteq K.
		\end{equation}
	First, take $x \in K$, say $x = \sum_{i \in I}\alpha_{i}x_{i}$, where $I$ is finite,  $\left\{\alpha_{i}\right\}_{i \in I} \subseteq \RR_{+}$, and $\left\{x_{i}\right\}_{i \in I} \subseteq C$; in addition, set $\alpha \coloneqq \sum_{i \in I} \alpha_{i}$. If $\alpha =0$, then, because $\left\{\alpha_{i}\right\}_{i \in I} \subseteq \RR_{+}$, we obtain $\left(\forall i \in I\right)~\alpha_{i}= 0$ and thus $x =0 \in D$; otherwise, we have  $\alpha > 0 $ and $x= \alpha \sum_{i\in I}\left(\alpha_{i} / \alpha \right)x_{i} \in \cone  \left( \conv{C}\right)\subseteq D$. Next, fix $y \in D$. Then  \cref{fact: cone}\cref{it: cone1} and \cref{eq: conv} yield the existence of $\lambda \in \RR_{++}$, a finite subset $\left\{\beta_{i}\right\}_{i \in J}$ of $\RR_{+}$, and a finite subset $\left\{x_{i}\right\}_{i \in J}$ of $C $ such that $y = \lambda \sum_{i \in J}\beta_{i} x_{i}$. In turn, if $\sum_{i \in J}\beta_{i} = 0$, then $\left(\forall i \in J\right)~\beta_{i} =0$ and so $y =0 \in E$; otherwise, $\sum_{i \in J}\beta_{i} >0$ and, upon setting $\beta \coloneqq \sum_{i \in J}\beta_{i}$, we get $y = \lambda \beta \sum_{i \in J}\left(\beta_{i} / \beta\right)x_{i} \in  \cone  \left(\conv{C}\right) \subseteq E$. Let us now prove that $E=F$. To do so, we infer from \cref{fact: cone}\cref{it: cone1} that 
		\begin{equation}
			E = \bigcup_{\lambda \in \RR_{++}}\lambda \left( \left\{0\right\} \cup \conv{C} \right) = \left\{0\right\} \cup \rb*{ \bigcup_{\lambda \in \RR_{++}} \lambda \conv{C}} = \left\{0\right\} \cup \cone \rb{\conv{C}} =F.
		\end{equation}
	Finally, take $z \in F$. If $z =0$, then clearly $z \in
	K$; otherwise, $z \in \cone \left(\conv{C}\right)$ and,
	by \cref{fact: cone}\cref{it: cone1} and \cref{eq: conv},
	there exist $\mu \in \RR_{++}$, a finite subset
	$\left\{\delta_{i}\right\}_{i \in T}$ of $\left]0,1
	\right]$, and a finite subset $\left\{x_{i}\right\}_{i
	\in T}$ of $C$ such that $z = \mu \sum_{i \in T}\delta_{i}x_{i} = \sum_{i\in T}\left(\mu \delta_{i}\right)x_{i} \in K$. Altogether,  \cref{eq: co-todo} holds.
	
	\cref{it: c2}: Since $K = {\cone}{ \left( {\conv}{\left(C \cup \left\{0\right\}\right) } \right) }$ by \cref{it: c1}, the conclusion thus follows from \cref{fact: cone}\cref{it: cone3}.
\end{proof}


\begin{example}[Lorentz cone]\label{eg: ic-cone}
	Let $\alpha  $ and $\beta$ be in $\RR_{++}$,  set 
	\begin{equation}\label{eq: ic-cone}
	\bK_{\alpha} \coloneqq \menge*{ \rb{x,\xi} \in \b{\HH} = \HH \oplus \RR }{ \norm{x} \leq \alpha \xi },
	\end{equation}
	and set 
	\begin{equation}
	\bC_{\alpha,\beta } \coloneqq \sphere{0}{\beta} \times \left\{\beta/\alpha \right\} \subseteq \b{\HH}.
	\end{equation}
	Then $\bK_{\alpha}$ is a nonempty closed convex cone in $\b{\HH}$ and 
	\begin{equation}\label{eq: ic-cone-gen}
	\bK_{\alpha} = \pos{\bC_{\alpha,\beta}} = \cone\rb{\conv {\bC_{\alpha,\beta}} } \cup \{\bzero \}.
	\end{equation}
\end{example}

\begin{proof}
	Since $\bK_{\alpha}$ is the epigraph of the function
	$\norm{} /\alpha $, which is  continuous, convex, and
	positively homogeneous\footnote{A function $f\colon \HH
	\to \left]-\infty,+ \infty \right]$ is \emph{positively
	homogeneous} if $\rb{\forall x \in \HH}\rb{\forall \lambda \in \RR_{++}}~ f\rb{\lambda x } = \lambda f\rb{x}.$}, we deduce from \cite[Proposition 10.2]{bauschke2017convex} that $\bK_{\alpha}$ is a nonempty closed convex cone in $\b{\HH}$. Next, let us establish \cref{eq: ic-cone-gen}. In view of \cref{lem: cone}\cref{it: c1}, it suffices to show that $\bK_{\alpha} = \cone\rb{\conv{\bC_{\alpha,\beta}}} \cup \{\bzero\}$. Towards this aim, let us first observe that, due to \cite[Exercise 3.2]{bauschke2017convex} and \cref{eg: co-S}, 
		\begin{equation}\label{eq: co-Cpa}
		\conv{\b{C}_{\alpha,\beta}} = \rb{\conv{\sphere{0}{\beta}}} \times \rb{\conv \{\beta /\alpha  \}} = \ball{0}{\beta} \times \{\beta/\alpha\}.
		\end{equation}
	Now set $\bK \coloneqq \cone\rb{\conv{\b{C}_{\alpha,\beta}}} \cup \{\bzero\}$, and take $\bx=\rb{x,\xi} \in \bK_{\alpha}$. If $\xi =0$, then  \cref{eq: ic-cone} yields $x=0$ and so $\bx = \bzero \in \b{K}$. Otherwise,  $\xi >0$ and we get from \cref{eq: ic-cone} that $\norm{\beta \rb{\alpha \xi}^{-1} x} = \beta \rb{\alpha \xi}^{-1}\norm{x} \leq \beta $ or, equivalently, $\beta \rb{\alpha \xi}^{-1}x \in \ball{0}{\beta}$; therefore, it follows from \cref{eq: co-Cpa} that  
	\begin{equation}
	\rb{x,\xi} = \frac{\alpha \xi }{\beta}\rb*{\frac{\beta}{\alpha \xi }x , \frac{\beta}{\alpha} } \in \frac{\alpha \xi }{\beta} \rb*{\ball{0}{\beta} \times \{\beta/\alpha\}} = \frac{\alpha \xi}{\beta}\conv{ \b{C}_{\alpha,\beta}} \subseteq \cone\rb{\conv{ \b{C}_{\alpha,\beta}}} \subseteq \b{K}.
	\end{equation}
Altogether, $\bK_{\alpha} \subseteq \bK$. Conversely, take $\by \in \cone\rb{\conv{ \b{C}_{\alpha,\beta}}} $. 
Then, by \cref{fact: cone}\cref{it: cone1} and \cref{eq: co-Cpa}, there exist $\lambda \in \RR_{++}$ and $y \in \ball{0}{\beta}$ such that $\by=\lambda\rb{y,\beta /\alpha} = \rb{\lambda y, \lambda\beta/\alpha}$. In turn, since $\norm{\lambda y} \leq \lambda\beta = \alpha \rb{ \lambda\beta /\alpha}$, we obtain $\by \in \b{K}_{\alpha}$. Hence $\cone\rb{\conv{ \b{C}_{\alpha,\beta}}} \subseteq \bK_{\alpha}$. This and the fact that $\bzero \in \bK_{\alpha}$ yield $\bK \subseteq \bK_{\alpha}$. Hence $\bK_{\alpha} = \bK$, as claimed. 
\end{proof}


Here is an improvement of \cite[Corollary 6.53]{bauschke2017convex}.

\begin{proposition}\label{pp: ipv}
	Let $C$ be a nonempty closed convex set in $\HH$. Suppose that there exists a nonempty closed subset $D$ of $C$ such that  $0 \notin D$ and that one of the following holds:  
	\begin{enumerate}\renewcommand{\labelenumi}{\rm (\alph{enumi})}\renewcommand{\theenumi}{\rm (\alph{enumi})}
		\item\label{it: ipv-a1} $\left(\cone{D}\right) \cup \left\{0\right\}  = \left(\cone{C}\right) \cup \left\{0\right\}$.
		\item\label{it: ipv-a2} $\ccone{D} = \ccone{C}$.
	\end{enumerate}
	Then the following hold: 
	\begin{enumerate}
		\item\label{it: ipv1} $\left(\cone{C}\right) \cup \left(\rec{C}\right) = \ccone{C}$.
		\item\label{it: ipv2} Suppose that $\rec{C} = \left\{0\right\}$. Then $\cone  \left(C\cup \left\{0\right\}\right)$ is closed.
	\end{enumerate}
\end{proposition}

\begin{proof}
	Let us first show that  \cref{it: ipv-a1}\ensuremath{\Rightarrow}\cref{it: ipv-a2}. To establish this,  assume that \cref{it: ipv-a1} holds. Then, since  $0 \in   \closu{\cone{D}}$ and $0 \in  \closu{\cone{C}}$ due to \cref{rem: clar}, we infer from \cref{fact: cone}\cref{it: cone2} that 
	\begin{equation}
	\ccone{D} =  \closu{\cone{D}} \cup\left\{0\right\} = \closu{\left(\cone{D}\right) \cup \left\{0\right\}}= \closu{\left(\cone{C}\right) \cup \left\{0\right\}} =  \closu{\cone{C}} \cup\left\{0\right\} = \ccone{C},
	\end{equation}
	which verifies the claim. Thus, it is enough to assume that \cref{it: ipv-a2} holds and to show that \cref{it: ipv1}\&\cref{it: ipv2} hold.
	
	\cref{it: ipv1}: Clearly $\cone{C}\subseteq \ccone{C}$. We now prove that $\rec{C}\subseteq \ccone{C}$. To this end, take $x \in \rec{C}$. Then \cite[Proposition 6.51]{bauschke2017convex} ensures the existence of sequences $\left(x_{n}\right)_{n \in \NN}$ in $C$ and $\left(\alpha_{n}\right)_{n \in \NN}$ in $\left] 0,1\right]$ such that  $\alpha_{n}x_{n}\to x$. Hence, because $\left\{\alpha_{n} x_{n}\right\}_{n \in \NN} \subseteq \cone{C}$ by \cref{fact: cone}\cref{it: cone1}, we deduce from \cref{fact: cone}\cref{it: cone2} that  $x \in \closu{\cone{C}} = \ccone{C}$. Thus $\left(\cone{C}\right) \cup \rec{C} \subseteq \ccone{C}$. Conversely, fix $y \in \ccone{C} = \ccone{D}$. It then follows from \cref{fact: cone}\cref{it: cone2} that $y \in \closu{\cone{D}}$, and therefore, in view of \cref{fact: cone}\cref{it: cone1}, there exist sequences $\left(\beta_{n}\right)_{n \in \NN}$ in $\RR_{++}$ and $\left(y_{n}\right)_{n \in \NN}$ in $D$ such that 
		\begin{equation}\label{eq: ipv-yn}
			\beta_{n} y_{n} \to y.
		\end{equation}
	After passing to subsequences and relabeling if necessary, we assume that  
		\begin{equation}\label{eq: ipv-bn}
			\beta_{n} \to \beta \in \left[0,+\infty\right].
		\end{equation}
	In turn, let us establish that $\beta \in \RR_{+}$ by contradiction: assume that $\beta =+\infty$. Then it follows from \cref{eq: ipv-yn} that $\norm{y_{n}} = \left(1/\beta_{n}\right)\norm{\beta_{n}y_{n}} \to 0$  or, equivalently, $y_{n}\to 0$. Hence, since $\left\{y_{n}\right\}_{n \in \NN} \subseteq D$, the closedness of $D$ asserts that $0 \in D$, which violates our assumption. Therefore $\beta \in \RR_{+}$, and this leads to two conceivable cases: 
		
		\hspace{\parindent}(A) $\beta =0$: Then, by \cref{eq: ipv-bn}, we can assume without loss of generality that $\left\{\beta_{n}\right\}_{n \in \NN} \subseteq \left]0,1 \right]$. In turn, since $\left\{y_{n}\right\}_{n \in \NN} \subseteq D \subseteq C$, we infer from \cref{eq: ipv-yn}\&\cref{eq: ipv-bn} and \cite[Proposition 6.51]{bauschke2017convex} that $y \in \rec{C}$.
		
		\hspace{\parindent}(B) $\beta > 0$: Then, in view of \cref{eq: ipv-yn}\&\cref{eq: ipv-bn}, $y_{n} = \left(1/\beta_{n}\right)\left(\beta_{n}y_{n}\right) \to y/\beta$. Therefore, because $\left\{y_{n}\right\}_{n \in \NN}\subseteq C$ and $C$ is closed, we obtain $y/\beta \in C$. Consequently, $y \in \beta C\subseteq \cone{C}$.
		
		To sum up, $\left(\cone{C}\right) \cup \left(\rec{C}\right) = \ccone{C}$, as announced.
		
	\cref{it: ipv2}: Since $C =\conv{C}$ by the convexity of $C$, we derive from \cref{it: ipv1} and \cref{lem: cone}\cref{it: c1} that $\ccone{C} = \left(\cone{C}\right) \cup \left\{0\right\} = \cone \left(C \cup \left\{0\right\}\right)$, which guarantees that $\cone\left(C\cup\left\{0\right\}\right)$ is closed.
\end{proof}

\begin{corollary}\label{cor: ipv}
	Let $C$ be a nonempty subset of $\HH$, and set $K \coloneqq \pos{C}$.   Suppose that $0 \notin \conv{C}$ and that $\conv{C}$ is weakly compact. Then $K$ is the smallest closed convex cone containing $C \cup \left\{0\right\}$.
\end{corollary}

	\begin{proof}
		According to \cref{lem: cone}\cref{it: c2}, it suffices to verify that $K$ is closed. Since $\conv{C}$ is weakly compact, it is weakly closed and bounded. In turn, on the one hand, since $\conv{C}$ is convex and weakly closed, we derive from \cite[Theorem 3.34]{bauschke2017convex} that $\conv{C}$ is closed. On the other hand, the boundedness of $\conv{C}$ guarantees that $\rec\rb{\conv{C} } = \{0\}$. Altogether, because $K = \cone \rb{\{0\} \cup \conv{C} }$ due to \cref{lem: cone}\cref{it: c1} and because $0 \notin \conv{C}$, applying \cref{pp: ipv}\cref{it: ipv2} to $\conv{C}$ (with the subset $D${\textemdash}as in the setting of \cref{pp: ipv}{\textemdash}being $\conv{C}$) yields the closedness of $K$, as required.
	\end{proof}


The following two examples provide instances in which the assumption of \cref{pp: ipv} holds.

\begin{example}
	Let $C$ be a nonempty closed convex subset of $\HH$ such that $C\smallsetminus \left\{0\right\} \neq \varnothing$. Suppose that there exists $\rho \in \RR_{++}$ satisfying 
	\begin{equation}\label{eq: eg-as-ipv}
	\left(\cone{C}\right) \cap \sphere{0}{\rho} \subseteq C,
	\end{equation}
	and set $D\coloneqq \left(\ccone{C}\right) \cap \sphere{0}{\rho}$. Then the following hold: 
	\begin{enumerate}
		\item\label{it: ipeg-11} $D$ is a nonempty closed subset of $C$ and $0 \notin D$.
		\item\label{it: ipeg-12} $\left(\cone{D}\right) \cup \left\{0\right\} = \left(\cone{C}\right) \cup \left\{0\right\}$.
	\end{enumerate}
\end{example}

	\begin{proof}
		\cref{it: ipeg-11}: The closedness of $D$ is clear. Next, since $0 \notin \sphere{0}{\rho}$, we have $0 \notin D$. In turn, since $C \smallsetminus  \left\{0\right\} \neq \varnothing$, we see that  $\varnothing\neq \ccone{C} \neq \left\{0\right\}$, and since $\ccone{C}$ is a cone, \cref{rem: clar} yields $D \neq \varnothing$. Finally, it follows from \cref{fact: cone}\cref{it: cone2}, \cref{lem: topo}\cref{it: topo1} (applied to $\cone{C}$), \cref{eq: eg-as-ipv}, and the closedness of $C$ that 
			\begin{math}
				D = \closu{\cone{C}} \cap \sphere{0}{\rho} = \closu{\left(\cone{C}\right) \cap \sphere{0}{\rho} } \subseteq \closu{C} = C,
			\end{math}
		as claimed.
		
		\cref{it: ipeg-12}: Because $D\subseteq C$, we get $\left(\cone{D}\right) \cup \left\{0\right\} \subseteq \left(\cone{C}\right) \cup \left\{0\right\}$. Conversely, take $x \in \cone{C}$. We then deduce from \cref{fact: cone}\cref{it: cone1} the existence of $\lambda \in \RR_{++}$ and  $y \in C$ such that $x = \lambda y$. If $y =0$, then $x =0 \in \left(\cone{D}\right) \cup \left\{0\right\}$. Otherwise, $\norm{y} \neq 0$ and, since $\rho y /\norm{y} \in  \left(\cone{C}\right) \cap \sphere{0}{\rho} \subseteq D$, we obtain $ x = \lambda y = \left(\lambda \norm{y}/\rho\right)\left(\rho y /\norm{y}\right) \in \cone{D}$. Therefore $\left(\cone{C}\right) \cup \left\{0\right\} \subseteq \left(\cone{D}\right) \cup \left\{0\right\}$, and the conclusion follows.
	\end{proof}

Before we present a new proof of the well-known
fact that finitely generated cones are closed (see
\cite[Theorem~19.1, Corollary~2.6.2 and the remarks following
Corollary~2.6.3]{rocky}), we make a few comments.

\begin{remark}\label{rm: fi-cone}
	Let $\set{x_{i}}{ i \in I}$ be a finite subset of $\HH$, and set $C\coloneqq \conv\set{x_{i}}{ i \in I }.$
	\begin{enumerate}
		\item\label{it: rmfc-1}  Since $C = \conv \cup_{i \in I}\set{x_{i}}{}$, \cite[Proposition 3.39(i)]{bauschke2017convex} implies that $C$ is compact, and so it is closed and bounded. In turn, the boundedness of $C$ gives $\rec{C}= \{0\}$.
		\item The geometric interpretation of the proof of \cref{eg: fi-cone} is as follows. If  $y$ lies in $C \smallsetminus \left\{0\right\}$, then the ray $\RR_{+}y$ must intersect a {\textquotedblleft}{face\textquotedblright} of $C$ that does not contain $0$. 
		\item \cref{eg: fi-cone} illustrates that the assumption of \cref{pp: ipv} is mild and covers the case of finitely generated cones.
	\end{enumerate}
\end{remark}

\begin{example}\label{eg: fi-cone}
	Let $\left\{x_{i}\right\}_{i \in I}$ be a finite subset of $\HH$ and set 
		\begin{equation}
			K \coloneqq \sum_{i \in I} \RR_{+} x_{i}.
		\end{equation}
	Then $K$ is the 
	smallest closed convex cone containing 
	$\left\{x_{i}\right\}_{i \in I} \cup \left\{0\right\}$.
\end{example}
	
	\begin{proof}
		We derive from \cref{eq: Minsum} and  \cref{lem: cone}\cref{it: c2} that   $K$ is the smallest convex cone in $\HH$ containing $\set{x_{i}}{i\in I} \cup \{0\}$. Therefore, it suffices to establish the closedness of $K$. Towards this goal,  we first infer from \cref{lem: cone}\cref{it: c1} (applied to $\set{x_{i}}{i \in I}$) that 
		\begin{equation}\label{eq: fc-K}
		K = \cone \rb{ \{0\} \cup  \conv  \set{x_{i}}{i\in I}   }.
		\end{equation} Furthermore, we  assume that 
			\begin{equation}\label{eq: fc-as1}
				\left\{x_{i}\right\}_{i \in I} \smallsetminus \left\{0\right\} \neq \varnothing,
			\end{equation}
		because otherwise the claim is trivial. In turn, set $C \coloneqq \conv \left\{x_{i}\right\}_{i \in I}$, 
			\begin{equation}
				\mathcal{I} \coloneqq \menge*{\varnothing \neq J \subseteq I}{0 \notin \conv \set{x_{i}}{i \in J} },
			\end{equation}
		and 
			\begin{equation}
				D \coloneqq \bigcup_{J \in \mathcal{I}} \conv \set{x_{i}}{i \in J} \subseteq C.
			\end{equation}
		Then, by \cref{eq: fc-as1}, $\mathcal{I}$ is nonempty,\footnote{We just need to pick a nonzero element $x_{j}$ of $\set{x_{i}}{i \in I}$ and set $J\coloneqq \left\{j\right\}$.} and thus, $0\notin D \neq \varnothing$. Moreover, $D$ is closed as a finite union of closed sets, namely $\fa{ \conv  \set{x_{i}}{i \in J} }{J \in \mathcal{I}}$. We now claim that 
			\begin{equation}\label{eq: fc-todo}
				\left(\cone{D}\right) \cup \{0\} = \left(\cone{C}\right) \cup \{0\}.
			\end{equation}
		To do so, it suffices to verify that $\left(\cone{C}\right) \cup \{0\} \subseteq \left(\cone{D}\right) \cup \{0\}$. Take $x \in \left(\cone{C}\right)\smallsetminus \{0\}$. Then \cref{fact: cone}\cref{it: cone1} ensures the existence of $\lambda \in \RR_{++}$ and  
			\begin{equation}\label{eq: fc-y}
				y \in C \smallsetminus \{0\}
			\end{equation}
		such that $x = \lambda y$. Since $y \in C = \conv \set{x_{i}}{i \in I}$, there exist a nonempty subset $J$ of $I$ and  
			\begin{equation}\label{eq: fc-a}
				\set{\alpha_{i}}{ i \in J } \subseteq \left]0,1 \right]
			\end{equation}
		such that $\sum_{i \in J}\alpha_{i} =1$ and $y = \sum_{i \in J} \alpha_{i} x_{i}$. If $J \in \mathcal{I}$, then $y \in \conv \set{x_{i}}{ i \in J}  \subseteq D$ and hence $x = \lambda y \in \cone{D}$. Otherwise, $0 \in \conv\set{x_{i}}{i \in J}$, and there thus exists $\set{\beta_{i}}{i \in J} \subseteq [0,1]$ such that $\sum_{i \in J} \beta_{i} x_{i} = 0$ and $J_{+} \coloneqq \menge{ i \in J }{\beta_{i} > 0  } \neq \varnothing$. In turn, set 
			\begin{equation}\label{eq: fc-g}
				\gamma \coloneqq \min_{ i \in J_{+}} \frac{\alpha_{i}}{\beta_{i}}
			\end{equation}
		and $\left(\forall i \in J\right)~ \delta_{i} \coloneqq \alpha_{i} - \gamma \beta_{i}$. By \cref{eq: fc-a} and \cref{eq: fc-g},   
			\begin{equation}\label{eq: fc-dt}
				\left(\forall i \in J\right) \quad \delta_{i} \geq 0.
			\end{equation}
		Now fix $j \in J_{+}$ such that $\alpha_{j} / \beta_{j} = \gamma$. Then we get $\delta_{j} = 0$ as well as $J \smallsetminus \{j\} \neq \varnothing$  (since otherwise, $J= \{j\}$ and $y= \alpha_{j} x_{j} = \gamma \beta_{j} x_{j} = 0$, which is absurd), and hence, 
			\begin{equation}\label{eq: fc-re-y}
				y = y - \gamma 0 = \sum_{i \in J} \alpha_{i} x_{i} - \gamma \sum_{i \in J}\beta_{i} x_{i} = \sum_{i \in J \smallsetminus \{j\} } \delta_{i} x_{i}.
			\end{equation}
		Therefore, in view of \cref{eq: fc-dt}, 
		\cref{eq: fc-re-y}, and \cref{eq: fc-y}, we must have $  \sum_{i \in J\smallsetminus\{j\}} \delta_{i} > 0$. In turn, if $J\smallsetminus\{j\} \in \mathcal{I}$, then set $\delta \coloneqq \sum_{i \in J\smallsetminus\{j\}} \delta_{i} $ and observe that  $y= \delta \sum_{i \in J\smallsetminus \{j\}} \rb{\delta_{i} / \delta} x_{i} \in \cone{D}$, which yields $x=\lambda y \in \cone{D}$. Otherwise, we reapply the procedure to $y = \sum_{i \in J \smallsetminus \{j\}} \delta_{i} x_{i}$ recursively until $y$ can be written as $y = \sum_{i \in J'}\gamma_{i}x_{i}$, where $J' \in \mathcal{I}$ and $\set{\gamma_{i}}{i \in J'} \subseteq \RR_{+}$ satisfying $\mu \coloneqq \sum_{i \in J'}\gamma_{i} > 0$. Consequently, $y = \mu \sum_{i \in J'} \rb{\gamma_{i} / \mu } x_{i} \in \cone{D}$, from which we deduce that $x = \lambda y \in \cone{D}$. Thus \cref{eq: fc-todo} holds, and since $\rec{C} = \left\{0\right\}$ (see \cref{rm: fi-cone}\cref{it: rmfc-1}), it follows from \cref{pp: ipv}\cref{it: ipv2} (applied to $C = \conv\set{x_{i}}{i \in I}$) and \cref{eq: fc-K} that $K$ is closed, as desired.
	\end{proof}


\section{Projection operators}

\label{s:proj}

Let $C$ be a nonempty subset of $\HH$. 
Recall that its \emph{distance function} is 
	\begin{equation}
		d_{C} \colon \HH\to \RR : x \mapsto \inf_{y \in C}\norm{x-y} 
	\end{equation}
while the corresponding 
\emph{projection operator} (or \emph{projector})
is the set-valued mapping 
	\begin{equation}
		P_{C} \colon \HH \to 2^{\HH} : x \mapsto \menge*{ u \in C}{ \norm{x-u}  = d_{C}{\left(x\right)} }.
	\end{equation}
Furthermore, if $C$ is closed and convex, then, for every $x \in \HH$, $P_{C}x $ is a singleton and we shall identify $P_{C}x$ with its unique element which is characterized by 
	\begin{equation}\label{eq: pr-C}
		P_{C}x \in C \quad \text{and} \quad \left(\forall y \in C\right)~\scal{y-P_{C}x}{x - P_{C}x} \leq 0;
	\end{equation}
see, for instance, \cite[Theorem 3.16]{bauschke2017convex}.
We start by recalling some known results.

\begin{fact}\label{fact: pr-K}
	Let $K$ be a nonempty closed convex cone in $\HH$, and let $x$ and $p$ be in $\HH$. Then 
	\begin{equation}
	p = P_{K}x \iff \left[\; p \in K,~ x-p \perp p,~\text{and~} x -p \in K^{\ominus}   \;\right].
	\end{equation}
\end{fact}

\begin{proof}
	See, e.g., \cite[Proposition 6.28]{bauschke2017convex}.
\end{proof}

Let us recall the celebrated Moreau decomposition for cones; see \cite{moreau1962decomposition}.

\begin{fact}[Moreau]\label{fact: Moreau}
	Let $K$ be a nonempty closed convex cone in $\HH$. Then 
	\begin{equation}
	\left(\forall x \in \HH\right) \quad x =P_{K}x + P_{K^{\ominus}}x \quad \text{and} \quad \norm{x}^{2} = d_{K}^{2}\rb{x} + d_{\pc{K}}^{2}\rb{x}.
	\end{equation}
\end{fact}

\begin{lemma}\label{lem: cpt}
	Let $K$ be a nonempty closed convex cone in $\HH$, and let $x \in \HH.$ Then the following hold: 
		\begin{enumerate}
			\item\label{it: cpt1} $P_{K}x \neq 0 \Leftrightarrow x \in \HH\smallsetminus \pc{K}.$
			\item\label{it: cpt2} Suppose that $P_{K}x \neq 0$. Let $\rho \in \RR_{++}$, and set $p \coloneqq \rb{\rho / \norm{P_{K}x} } P_{K}x$. Then 
				\begin{equation}\label{eq: cpt-dst}
					\norm{x-p} = \sqrt{ d_{K}^{2}{\rb{x}} + \rb{ \norm{P_{K}x} -\rho }^{2} }.
				\end{equation}
		\end{enumerate}
\end{lemma}

	\begin{proof}
		\cref{it: cpt1}: We deduce from \cref{fact: Moreau} that 
			\begin{math}
				x \in \pc {K} \Leftrightarrow x - P_{\pc{K}}x =0 \Leftrightarrow P_{K}x  =0,
			\end{math}
		and the claim follows.

		\cref{it: cpt2}: Set $\beta \coloneqq \rho / \norm{P_{K}x}$. Then, because $x-P_{K}x \perp P_{K}x$ by  \cref{fact: pr-K},  the Pythagorean identity implies that  
		\begin{subequations} 
			\begin{align}
				\norm{x-p}^{2} 
					& = \norm{x- \beta P_{K}x }^{2} \\
					& = \norm{ \left(x-P_{K}x \right) + \left(1-\beta\right)P_{K}x }^{2} \\
					& = \norm{x-P_{K}x}^{2}  + \left(1-\beta\right)^{2} \norm{P_{K}x}^{2}\\
					& = d_{K}^{2}{\left(x\right)} + \left(1-\rho / \norm{P_{K}x}\right)^{2}\norm{P_{K}x}^{2} \\
					& = d_{K}^{2}{\left(x\right)} + \left( \norm{P_{K}x} - \rho \right)^{2},
			\end{align}
		\end{subequations}
		and thus \cref{eq: cpt-dst} holds.
	\end{proof}

We now turn to projectors onto subsets of spheres. 

\begin{lemma}\label{lem: pil1}
	Let $C$ be a nonempty subset of $\HH$ consisting of
	vectors of equal norm, let $x \in \HH$, and let $p \in
	C$. 
	Then the following hold\footnote{The characterization
	in item \cref{it:p11} plays also a role in 
	\cite[Corollaries 2 and 3]{ferreira2013projections}.}: 
		\begin{enumerate}
			\item\label{it:p11} $p \in P_{C}x \Leftrightarrow \scal{x}{p} = \max \scal{x}{C}$.
			\item\label{it:p12} $P_{C}x \neq \varnothing$ if and only if $\scal{x}{ \cdot \,} $ achieves its supremum over $C$.
		\end{enumerate}
\end{lemma}

	\begin{proof}
		\cref{it:p11}: Indeed, since $p \in C$ and $\rb{\forall y \in C}~\norm{y} = \norm{p}$ by our assumption, we see that 
			\begin{subequations}
				\begin{align}
					p \in P_{C}x 
						&\iff \rb{\forall y \in C}~ \norm{x-p}^{2} \leq \norm{x-y}^{2} \\
						&\iff \rb{\forall y \in C}~ -2\scal{x}{p} \leq -2\scal{x}{y} \\
						&\iff \rb{\forall y \in C}~ \scal{x}{y} \leq \scal{x}{p} \\
						&\iff \scal{x}{p} = \max \scal{x}{C},
				\end{align}
			\end{subequations}
		which verifies the claim.
		
		\cref{it:p12}: This follows  from \cref{it:p11}.
	\end{proof}

The following example provides an instance in which $P_{C}x \neq \varnothing$, where $C$ and $x$ are as in \cref{lem: pil1}.

\begin{example}
	Consider the setting of \cref{lem: pil1} and suppose, in addition, that $C$ is weakly closed. Then $P_{C}x \neq \varnothing$.
\end{example}	
	
	\begin{proof}
		Since, by assumption, $C$ is bounded and since $C$ is weakly closed, we deduce that $C$ is weakly compact (see, for instance, \cite[Lemma 2.36]{bauschke2017convex}). Therefore, because $\scal{x }{\cdot\,} $ is weakly continuous, its supremum over $C$ is achieved, and the assertion therefore follows from \cref{lem: pil1}\cref{it:p12}.
	\end{proof}


\begin{lemma}\label{lem: pil2}
	Let $C$ be a nonempty subset of $\HH$,  let $\beta \in
	\RR_{++}$, and let $u \in \pos{C}$, say $u = \sum_{i \in
	I}\alpha_{i} x_{i}$, where $\set{\alpha_{i}}{i \in I} $
	and $\set{x_{i}}{i \in I}$ are finite subsets of $\RR_{+}$ and $C$, respectively. Suppose that $\norm{u} = \beta$ and that   $\left(\forall y \in C\right)~ \norm{y}  = \beta$. Then the following hold: 
		\begin{enumerate}
			\item\label{it: p21} $\sum_{i \in I} \alpha_{i} \geq 1$. 
			\item\label{it: p22} Let $x \in \HH$, and set $\kappa \coloneqq \sup \scal{x}{C}$. Suppose that $\kappa \in \left] -\infty ,0 \right]$ and that  $ \kappa\leq \scal{x}{u} $. Then the following hold: 
				\begin{enumerate}
					\item\label{it: p22a} $P_{C}x \neq \varnothing$ and $\scal{x}{u} = \max \scal{x}{C} = \kappa$.
					\item\label{it: p22b}   $u \in \sphere{0}{\beta}\cap \cone\rb{\conv P_{C}x}$.
					\item\label{it: p22c} Suppose that $\kappa < 0$. Then $u \in P_{C}x$. 
				\end{enumerate}
		\end{enumerate}
\end{lemma}

	\begin{proof}
		\cref{it: p21}: Since, by assumption, $\left(\forall i\in I\right)~ \norm{x_{i}} =  \beta$, it follows from the triangle inequality that 	
			\begin{equation}
				 \beta = \norm{u} = \norm*{\sum_{i \in I} \alpha_{i} x_{i} } \leq \sum_{i\in I}\alpha_{i} \norm{x_{i}} = \beta \sum_{i \in I}\alpha_{i}.
			\end{equation}
		Therefore, because $\beta >0$, we obtain $\sum_{i \in I} \alpha_{i} \geq 1$.
		
		\cref{it: p22}:   Let us first establish that 
			\begin{equation}\label{eq: p2-con}
				\rb{\forall i \in I} \quad x_{i} \in C\smallsetminus P_{C}x \implies \alpha_{i} =0 
			\end{equation}
		by contradiction: assume that there exists $i_{0} \in I$ such that 
			\begin{equation}\label{eq: xio-c}
				x_{i_{0}} \in C\smallsetminus P_{C}x
 			\end{equation}
 		but that 
 			\begin{equation}\label{eq: aio-c}
	 			\alpha_{i_{0}} >0.
 			\end{equation}
 		Then, because the vectors in $C$ are of equal norm, we deduce from \cref{lem: pil1}\cref{it:p11} and \cref{eq: xio-c} that $\scal{x}{x_{i_{0}}} < \sup \scal{x}{C} = \kappa$, and so, by \cref{eq: aio-c}, $\alpha_{i_{0}} \scal{x}{x_{i_{0}}} < \alpha_{i_{0}} \kappa$. Hence, since 
 			\begin{equation}\label{eq: p22-ass}
	 			\left\{ 
		 				\begin{array}{l}
		 				\kappa \leq 0, \\
		 				\kappa \leq \scal{x}{u} , \\
	 					\rb{\forall i \in I } ~ 0\leq  \alpha_{i}   \text{~and~}   \scal{x}{x_{i}} \leq \kappa,
	 				\end{array}
	 				\right.
 			\end{equation}
 		it follows from \cref{it: p21} that 
 			\begin{subequations}
	 			\begin{align}
	 				\kappa \leq \scal{x}{u} 
	 					& = \sum_{i \in I} \alpha_{i} \scal{x}{x_{i}} \\
	 					& =\alpha_{i_{0}} \scal{x}{x_{i_{0}}} +  \sum_{i \in I \smallsetminus \{i_{0}\}}\alpha_{i} \scal{x}{x_{i}}  \\
	 					& < \alpha_{i_{0}} \kappa + \sum_{i \in I \smallsetminus \{i_{0}\}} \alpha_{i} \kappa  \\
	 					& = \kappa \sum_{i \in I} \alpha_{i} \\
	 					& \leq \kappa,
	 			\end{align}
 			\end{subequations}
 		and we thus arrive at a contradiction, namely $\kappa < \kappa$. Therefore, \cref{eq: p2-con} holds.
 		
 		\hspace{\parindent}\cref{it: p22a}: If $P_{C}x$ were empty, then \cref{eq: p2-con} would yield $\rb{\forall i\in I}~\alpha_{i}=0$ and it would follow that $u=0$ or, equivalently, $\beta =\norm{u} = 0$, which is absurd. Thus $P_{C}x \neq \varnothing$, and so \cref{lem: pil1}\cref{it:p11} implies that $\kappa =\max \scal{x}{C}$. Furthermore, we infer from \cref{eq: p22-ass} and \cref{it: p21} that 
 			\begin{equation}
	 				\kappa \leq \scal{x}{u} 
	 			 = \sum_{i \in I} \alpha_{i} \scal{x}{x_{i}} \leq \sum_{i\in I}\alpha_{i} \kappa = \kappa \sum_{i \in I}\alpha_{i} \leq \kappa,
 			\end{equation}
 		and the latter assertion follows.
 		
 		\hspace{\parindent}\cref{it: p22b}: In the remainder, since $u \neq 0$,  appealing to \cref{eq: p2-con}, we assume without loss of generality that 
 			\begin{equation}\label{eq: p2-WLOG}
	 			\rb{\forall i\in I}\quad x_{i} \in P_{C}x
 			\end{equation}
 		and that $\rb{\forall i \in I}\rb{\forall j \in I}~i\neq j \Rightarrow x_{i}\neq x_{j}$. Hence, upon setting $\alpha \coloneqq \sum_{i \in I} \alpha_{i} \geq 1$, we deduce from \cref{eq: p2-WLOG} that 
 			\begin{equation}
	 			u= \alpha \sum_{i\in I}\frac{\alpha_{i}}{\alpha} x_{i} \in \alpha \conv{P_{C}x} \subseteq \cone \rb{\conv{P_{C}x } }.
 			\end{equation}
 		Consequently, since $\norm{u} = \beta$, the claim follows.
 		
 		\hspace{\parindent}\cref{it: p22c}: Invoking \cref{lem: pil1}\cref{it:p11} and \cref{eq: p2-WLOG}, we get $\rb{\forall i\in I}~ \scal{x}{x_{i}} = \max \scal{x}{C} = \kappa$. Thus, by \cref{it: p22a}, 
 			\begin{math}
	 			\kappa = \scal{x}{u} = \sum_{i \in I} \alpha_{i} \scal{x}{x_{i}} = \kappa \sum_{i \in I}\alpha_{i},
 			\end{math}
 		and since $\kappa \neq 0$, it follows that $\sum_{i\in I}\alpha_{i} =1$. To summarize, we have 
 			\begin{equation}
	 			\left\{ 
		 			\begin{array}{l}
		 				u = \sum_{i \in I}\alpha_{i} x_{i}, \\
		 				\rb{ \forall i \in I} ~ \norm{x_{i} } = \norm{u} =\beta , \\
		 				\set{\alpha_{i}}{ i \in I } \subseteq \RR_{+} \text{~satisfying~} \sum_{i \in I}\alpha_{i} =1, \\
		 				\rb{\forall i\in I}\rb{\forall j \in I}~i\neq j \Rightarrow x_{i} \neq x_{j}.
	 				\end{array}
	 				\right.
 			\end{equation}
 		\cref{lem: numb}\cref{it: vt2} and \cref{eq: p2-WLOG} therefore imply that $\rb{\exists i\in I}~ u=x_{i} \in P_{C}x$, as desired.
	\end{proof}

The following example shows that the conclusion of \cref{lem: pil2}\cref{it: p22c} fails if the assumption that  $u \in \pos{C}$ is omitted.

\begin{example}
	Suppose that $\HH = \RR^{3}$ and that $\rb{e_{1},e_{2}, e_{3}}$ is the canonical orthonormal basis of $\HH$. Set $C\coloneqq \{e_{1},e_{2} \}$, $x\coloneqq (-1,-1,0)$, and $u\coloneqq (1/2, 1/2, \sqrt{2}/2)$. Then $u$ is not a conical combination of elements of $C$ and, as in the assumption of \cref{lem: pil2}, $(\beta ,\kappa) = (1,-1)$. Moreover, a simple computation gives $\norm{u}=1$, $\scal{x}{u}=-1 =\kappa$, and $ \norm{x-e_{1}} = \norm{x-e_{2}}=\sqrt{5}$. Hence,  $P_{C}x = C\neq \varnothing$ while $u \notin C$.
\end{example}

\section{Projectors onto sets of real symmetric matrices}\label{sect: SN}

\label{s:Sn}

In this section, $N$ is a strictly positive integer, and suppose
that $\HH = \bbS^{N}$ is the Hilbert space of real symmetric matrices endowed with the scalar product 
	\begin{math}
		\scal{}{ } \colon \rb{A,B} \mapsto \tra\rb{AB},
	\end{math}
where $\tra$ is the trace function; the associated norm is the Frobenius norm $\fnorm{}$. The closed convex cone of positive semidefinite symmetric matrix in $\HH$ is denoted by $\bbS_{+}^{N}$, and the set of orthogonal matrices of size $N \times N$ is $\bbU^{N} \coloneqq \menge{U \in \RR^{N\times N}}{ U U^{\T} = \Id }$, where $\Id$ is the identity matrix of $\RR^{N\times N}$. Next, for every $x = \rb{\xi_{1},\ldots,\xi_{N}} \in \RR^{N}$, set $x_{+}\coloneqq \rb{\max  \{\xi_{i},0\}  }_{1\leq i\leq N}$ and define $\Diag{x}$ to be  the diagonal matrix whose, starting from the upper left corner, diagonal entries are $\xi_{1},\ldots,\xi_{N}$. Now, for every $A \in \HH$, the eigenvalues of $A$ (not necessarily distinct) are denoted by $\rb{\lambda_{i}\rb{A} }_{1\leq i\leq N}$ with the convention that $\lambda_{1}\rb{A} \geq \cdots \geq \lambda_{N}\rb{A}$. In turn, the mapping $\lambda \colon \HH \to \RR^{N} : A \mapsto \rb{\lambda_{1}\rb{A},\ldots,\lambda_{N}\rb{A} }$ is well defined. Finally, the Euclidean scalar product and norm of $\RR^{N}$ are respectively denoted by $\scal{}{}$ and $\norm{}$.

\begin{remark}\label{rm: fnorm}
	Let $A \in \HH ,~ U \in \bbU^{N}, \text{~and~} x \in \RR^{N}$. Then it is straightforward to verify that 
		\begin{equation}
			\fnorm{ U A U^{\T} } = \fnorm{A } = \norm{\lambda\rb{A} }
		\end{equation}
	and that 
		\begin{equation}
			\fnorm{ U \rb{\Diag{x} } U^{\T} } = \fnorm{\Diag {x} } = \norm{x}.
		\end{equation}
\end{remark}

\begin{lemma}\label{fact: pr-psd}
	Set $K \coloneqq \bbS_{+}^{N}$. Let $A \in \HH$, and  let $U \in \bbU^{N}$ be such that $A = U \rb{\Diag \lambda\rb{A} } U^{\T}$. Then $P_{K}A =  U \rb{\Diag \rb{\lambda\rb{A}}_{+} } U^{\T} $ and $\fnorm{P_{K}A } = \norm{\rb{\lambda\rb{A}}_{+} }$.
\end{lemma}
	
	\begin{proof}
It is well known 
that  $P_{K}A =  U \rb{\Diag \rb{\lambda\rb{A}}_{+} } U^{\T}$
(see, e.g., \cite[Theorem~A1]{lewis2008alternating}  
or \cite[Example~29.32]{bauschke2017convex}). 
In turn, since $U \in \bbU^{N}$, it follows from \cref{rm: fnorm}
that $\fnorm{P_{K}A } = \norm{\rb{\lambda\rb{A}}_{+} }$.
	\end{proof}

\begin{fact}[Theobald] {\rm (See \cite{theobald1975inequality}.)} \label{fact: Theo}
	 Let $A$ and $B$ be in $\HH$. Then the following hold: 
	\begin{enumerate}
		\item\label{it: Theo1} $\scal{A}{B}\leq \scal{\lambda\rb{A} }{  \lambda\rb{B}}.$
		\item\label{it: Theo2} $\scal{A}{B} =  \scal{\lambda\rb{A} }{  \lambda\rb{B}}$ if and only if there exists $U \in \bbU^{N}$ such that $A = U \rb{\Diag \lambda\rb{A}  } U^{\T}$ and $B = U \rb{\Diag \lambda\rb{B}  } U^{\T}$.
	\end{enumerate}
\end{fact}

\begin{lemma}\label{lem: SN}
	Let $\rho \in \RR_{++}$, and set 	
	\begin{equation}\label{eq: Crho-dfn}
	\sC_{\rho} \coloneqq \menge*{ A \in \bbS_{+}^{N} }{\rank{A}=1 \text{~and~} \fnorm{A}=\rho }.
	\end{equation}
	Then the following hold: 
	\begin{enumerate}
		\item\label{it: psd-gen} $\bbS_{+}^{N} = \pos{\sC_{\rho}}$.
		\item\label{it: Crho} \begin{math} \sC_{\rho} = \menge{  A \in \HH }{ \rb{\exists U \in \bbU^{N} }~A = U\rb{\Diag\rb*{\rho,0,\ldots,0}} U^{\T}  }.\end{math} 
		\item\label{it: pr-Crho} Let $A \in \HH$. Then $\max \scal{A}{\sC_{\rho}} = \rho \lambda_{1}\rb{A}$ and 
			\begin{equation}\label{eq: pr-Crho}
				P_{\sC_{\rho}}A  = \menge*{ U \rb{\Diag\rb*{\rho,0,\ldots,0}} U^{\T} }{ U \in \bbU^{N} \text{~such that~} A = U \rb{\Diag \lambda\rb{A} } U^{\T} } \neq \varnothing.
			\end{equation}
	\end{enumerate}
\end{lemma}
	
	\begin{proof}
		\cref{it: psd-gen}: Set $I \coloneqq \{1,\ldots,N\}$, and let $\fa{e_{i}}{ i\in I}$ be the canonical orthonormal basis of $\RR^{N}$. First, since $\sC_{\rho}\cup \{0\} \subseteq \bbS_{+}^{N}$ and $\bbS_{+}^{N}$ is a convex cone, we infer from \cref{lem: cone}\cref{it: c2} that  
			\begin{math}
				  \pos{\sC_{\rho}}\subseteq \bbS_{+}^{N}.
			\end{math}		
		Conversely, take $A \in \bbS_{+}^{N}$, and let $U \in \bbU^{N}$ be such that  $A = U\rb{ \Diag\lambda\rb{A} } U^{\T}$; in addition, set 	
			\begin{math}
				\rb{\forall i \in I}~ D_{i} \coloneqq \Diag \rb{ \rho e_{i} }\in \bbS_{+}^{N}.
			\end{math}
		Then, for every $i \in I$,  since  \begin{math}
			 \rank D_{i} =1 
		\end{math}
		and $\fnorm{U D_{i} U^{\T} } = \fnorm{D_{i}} = \norm{\rho e_{i}} =\rho$, we get from \cref{eq: Crho-dfn} that  $U D_{i}U^{\T} \in   \sC_{\rho} $. In turn, because $\{\lambda_{i}\rb{A} \}_{i \in I} \subseteq \RR_{+}$ and 
			\begin{equation}
				A = U\rb{ \Diag \lambda\rb{A} } U^{\T} = U \rb*{ \sum_{i \in I} \Diag \rb{\lambda_{i}\rb{A} e_{i} }  } U^{\T} = U \rb*{ \sum_{i \in I} \frac{\lambda_{i}\rb{A}}{\rho} D_{i}  } U^{\T}  = \sum_{i \in I} \frac{\lambda_{i}\rb{A} }{\rho} \rb{ U D_{i} U^{\T} },
			\end{equation}
		we deduce that $A \in \pos{\sC_{\rho}}$. Hence, $\bbS_{+}^{N} = \pos{\sC_{\rho}}$.

		\cref{it: Crho}: 
			Recall that, if $A $ is a matrix of rank $r$ in $\bbS_{+}^{N}$, then 
			\begin{equation}\label{eq: ei-Rankr}
			\lambda_{1}\rb{A} \geq \cdots \geq \lambda_{r}\rb{A} > \lambda_{r+1}\rb{A}=\cdots \lambda_{N}\rb{A} = 0.
			\end{equation}
		Now set $\sD \coloneqq \menge{  A \in \HH }{ \rb{\exists U \in \bbU^{N} }~A = U\rb{\Diag\rb*{\rho,0,\ldots,0}} U^{\T}  }$. First, take $A \in \sC_{\rho}$, and let $U \in \bbU^{N} $ be such that $A = U \rb{\Diag\lambda\rb{A} } U^{\T}$. Then, since $\rank{A}=1$ and $A \in \bbS_{+}^{N}$, it follows from \cref{eq: ei-Rankr} that $\lambda\rb{A} = \rb{\lambda_{1}\rb{A} , 0,\ldots, 0}$ and $\lambda_{1}\rb{A} >0$; therefore, because $\fnorm{A}=\rho$, we obtain 
			\begin{math}
				\rho = \fnorm{A}  = \norm{\lambda\rb{A}} = \lambda_{1}\rb{A}.
			\end{math}
		Hence, $A = U \rb{\Diag \rb{\lambda_{1}\rb{A} , 0,\ldots, 0} } U^{\T}  = U \rb{ \Diag\rb{\rho,0,\ldots, 0} } U^{\T}$, which yields $A \in \sD$. Conversely, take $B \in \sD$, say $B = V\rb{\Diag\rb{\rho, 0 ,\ldots, 0}} V^{\T}$, where $V \in \bbU^{N}$. Then, since $\rho >0$, we have $B \in \bbS_{+}^{N}$. Next, on the one hand, because $V$ is nonsingular and $\rho \neq 0$, we have $\rank{B} = \rank \Diag\rb{\rho, 0 ,\ldots, 0}  =1$. On the other hand, since $V \in \bbU^{N}$, it follows that $\fnorm{B} = \fnorm{ V\rb{\Diag\rb{\rho, 0 ,\ldots, 0}} V^{\T} } = \norm{\rb{\rho, 0 ,\ldots, 0} } = \rho$. Altogether, $B \in \sC_{\rho}$, which completes the proof.
		
		\cref{it: pr-Crho}: First, it follows from \cref{it: Crho} that  
			\begin{equation}\label{eq: Crho-ei}
				\rb{\forall B \in \HH} \quad  B \in \sC_{\rho} \iff \lambda\rb{B} = \rb{\rho,0,\ldots,0}.
			\end{equation}
		Next, denote  the right-hand set of \cref{eq: pr-Crho} by $\sD$. Then, by \cref{it: Crho}, $\varnothing \neq \sD \subseteq \sC_{\rho}$. Now, 
		for every $B \in \sC_{\rho}$, since $\lambda\rb{B} = \rb{\rho,0,\ldots, 0}$, we infer from \cref{fact: Theo}\cref{it: Theo1} that $\scal{A}{B} \leq \scal{\lambda\rb{A}}{\lambda\rb{B} } = \rho \lambda_{1}\rb{A}$. Thus, $\sup\scal{A}{\sC_{\rho}} \leq \rho \lambda_{1}\rb{A}$. Furthermore, by   \cref{eq: Crho-ei}, \cref{fact: Theo}\cref{it: Theo2}, and the very definition of $\sD$, we see that
		\begin{subequations}  
			\begin{align}
				\rb{\forall B \in \sC_{\rho}} \quad 
				\scal{A}{B} = \rho \lambda_{1}\rb{A} 
				& \iff \scal{A}{B}  = \scal{\lambda\rb{A}}{\lambda\rb{B}} \\
				& \iff  \rb{\exists U \in \bbU^{N} } ~ \left\{ \begin{array}{l}
				A = U\rb{ \Diag\lambda\rb{A} } U^{\T}, \\
				B = U \rb{ \Diag\lambda\rb{B} } U^{\T}
				\end{array} \right.   \\
				& \iff  \rb{\exists U \in \bbU^{N} } ~ \left\{ \begin{array}{l}
				A = U\rb{ \Diag\lambda\rb{A} } U^{\T}, \\
				B = U \rb{ \Diag\rb{\rho,0,\ldots,0} } U^{\T}
				\end{array} \right. \\
				& \iff B \in \sD.
			\end{align} 
		\end{subequations}
	Therefore, because $\sD\neq \varnothing$, we deduce that $\max \scal{A}{\sC_{\rho}} = \rho \lambda_{1}\rb{A}$ and \begin{math}
		\rb{\forall B \in \sC_{\rho}}~\scal{A}{B} = \max \scal{A}{\sC_{\rho}} \Leftrightarrow B \in \sD.
	\end{math} Consequently, since the matrices in $\sC_{\rho}$ are of equal norm by \cref{eq: Crho-dfn}, we derive from \cref{lem: pil1}\cref{it:p11} that $P_{\sC_{\rho}} A = \sD$, as desired.
	\end{proof}

\section{Projectors onto cones generated by orthonormal sets}
\label{sect: spe-cone}

\label{s:fg}

We start with a conical version 
of \cite[Example~3.10]{bauschke2017convex}.

\begin{theorem}\label{thm: pr-K-ei}
	Let  $\set{e_{i}}{i \in I}$ be a nonempty finite orthonormal subset of $\HH$, set 
		\begin{equation}\label{eq: K-ei}
			K \coloneqq \sum_{i \in I}\RR_{+} e_{i},
		\end{equation}
	and let $x \in \HH$. Then $K$ is a nonempty closed convex cone in $\HH$, 
		\begin{equation}\label{eq: pr-ei}
			P_{K}x = \sum_{i \in I} \max\left\{\scal{x}{e_{i}},0 \right\} e_{i}, \quad \text{and} \quad d_{K}\rb{x} = \sqrt{\norm{x}^{2} - \sum_{i\in I} \rb{\max \left\{\scal{x}{e_{i}} ,0 \right\} }^{2}  }.
		\end{equation} 
\end{theorem}

	\begin{proof}
		We first infer from \cref{eg: fi-cone} that $K$ is a nonempty closed convex cone. Thus, it is enough to verify \cref{eq: pr-ei}. To this end, set 
			\begin{equation}\label{eq: dfn-ai}
				\rb{\forall i \in I} \quad \alpha_{i} \coloneqq \max \left\{\scal{x}{e_{i}},0\right\} \in \RR_{+}
			\end{equation}
		and  
			\begin{equation}\label{eq: p-ei}
				p \coloneqq \sum_{i \in I} \alpha_{i} e_{i}.
			\end{equation}
		Then, by \cref{eq: dfn-ai}\&\cref{eq: p-ei}\&\cref{eq: K-ei}, we have $p \in K$, and by assumption, we get 
			\begin{equation}\label{eq: N-p}
				\norm{p}^{2} = \norm*{\sum_{i\in I} \alpha_{i} e_{i} }^{2} = \sum_{i \in I}\alpha_{i}^{2}.
			\end{equation}
		Furthermore,  \cref{eq: dfn-ai} implies that 
			\begin{equation}
				\rb{\forall i \in I} \quad \left[\; \alpha_{i} = \scal{x}{e_{i}}\text{~or~} \alpha_{i} =0 \;\right] \iff \alpha_{i} \rb{ \scal{x}{e_{i}} - \alpha_{i}  } =0 \iff \alpha_{i} \scal{x}{e_{i}} = \alpha_{i}^{2},
			\end{equation}
		and therefore, we get from \cref{eq: p-ei} that 
			\begin{equation}\label{eq: XP}
				\scal{x}{p}  = \scal*{x}{\sum_{i \in I} \alpha_{i}e_{i} } = \sum_{i \in I} \alpha_{i} \scal{x}{e_{i}} = \sum_{i \in I}\alpha_{i}^{2}.
			\end{equation}
		In turn, on the one hand, \cref{eq: N-p} and \cref{eq: XP} yield $\scal{x-p}{p} = \scal{x}{p}- \norm{p}^{2} = 0$. On the other hand, invoking   \cref{eq: p-ei}, \cref{eq: dfn-ai}, and our hypothesis, we deduce that 
			\begin{equation}
				\rb{\forall i \in I} \quad \scal{x-p}{e_{i}}  = \scal{x}{e_{i}} - \scal*{\sum_{j \in I} \alpha_{j}e_{j} }{e_{i}} = \scal{x}{e_{i}} - \alpha_{i} \leq 0,
			\end{equation}
		and hence, by \cref{eq: K-ei}, $x-p \in \pc{K}$. Altogether, we conclude that  $P_{K}x = p  = \sum_{i \in I} \max \left\{\scal{x}{e_{i}},0\right\} e_{i}$ via \cref{fact: pr-K}. Consequently, \cref{eq: N-p}\&\cref{eq: XP}\&\cref{eq: dfn-ai} give 
			\begin{equation}
				d_{K}^{2}\rb{x} = \norm{x-p}^{2} = \norm{x}^{2} - 2\scal{x}{p} + \norm{p}^{2} = \norm{x}^{2} - \sum_{i \in I} \alpha_{i}^{2} = \norm{x}^{2} - \sum_{i \in I} \rb{\max \left\{ \scal{x}{e_{i}},0 \right\} }^{2},
			\end{equation}
		which completes the proof.
	\end{proof}

\begin{remark}
	Here are a few comments concerning \cref{thm: pr-K-ei}.
		\begin{enumerate} 
			\item In the setting of \cref{thm: pr-K-ei}, suppose that $\left\{e_{i}\right\}_{i \in I}$ is a singleton, say $e$. Then $K= \RR_{+} e$ is a ray and  \cref{eq: pr-ei} becomes 
			\begin{equation}
		 		  P_{K}x = \max \left\{ \scal{x}{e},0 \right\} e  \quad \text{and} \quad d_{K}{\left(x\right)} = \sqrt{\norm{x}^{2} - \left(\max \left\{\scal{x}{e},0 \right\} \right)^{2} },
			\end{equation}
			which is precisely the formula for projectors onto rays (see, e.g., \cite[Example 29.31]{bauschke2017convex}).
			\item Consider the setting of \cref{thm: pr-K-ei}. Suppose that $N$ is a strictly positive integer, that $I = \left\{1,\ldots, N \right\}$, that $\HH= \RR^{N}$, and that $\left(e_{i}\right)_{i \in I}$ is the canonical orthonormal basis of $\HH$. Then $K = \RR_{+}^{N}$ is the positive orthant in $\HH$. Now take $x = \left(\xi_{i}\right)_{i \in I} \in \HH$. In the light of \cref{eq: pr-ei}, since $\left(\forall i \in I\right)~ \scal{x}{e_{i}} = \xi_{i}$, we retrieve the well-known formula
			\begin{equation}\label{eq: pr-orthant}
			P_{K} x = \left( \max \left\{\xi_{i},0\right\} \right)_{i \in I};
			\end{equation}
			see, for instance, \cite[Example 6.29]{bauschke2017convex}. Moreover, upon setting $I_{-} \coloneqq \menge{ i \in I}{\xi_{i} < 0}$, we derive from \cref{eq: pr-ei} that 
			\begin{equation}
			d_{ K }{\left(x\right)} = \sqrt{\norm{x}^{2} - \sum_{i \in I} \left(\max \left\{\xi_{i},0\right\} \right)^{2}} = \sqrt{\sum_{i\in I} \xi_{i}^{2} - \sum_{i \in I \smallsetminus I_{-} } \xi_{i}^{2}  } = \sqrt{ \sum_{i \in I_{-} } \xi_{i}^{2}  }
			\end{equation}
			with the convention that $\sum_{i \in \varnothing} \xi_{i}^{2} =0$.
		\end{enumerate}
\end{remark}

\begin{corollary}
	Let $\set{e_{i}}{i \in I}$ be a nonempty finite orthonormal subset of $\HH$. Set 
		\begin{equation}
			K \coloneqq \menge{ y \in \HH }{ \rb{\forall i \in I}~ \scal{y}{e_{i}} \leq 0 },
		\end{equation}
	and let $x \in \HH$. Then $K$ is a nonempty closed convex cone in $\HH$, 
		\begin{equation}\label{eq: pr-Kei-po}
			P_{K}x = x - \sum_{i \in I} \max\left\{\scal{x}{e_{i}},0 \right\} e_{i}, \quad \text{and} \quad d_{K}\rb{x} = \sqrt{ \sum_{i \in I} \rb{  \max\left\{\scal{x}{e_{i}},0 \right\}}^{2} }.
		\end{equation}
\end{corollary}

	\begin{proof}
		Since  
			\begin{equation}\label{eq: pc-ei}
				K = \bigcap_{i\in I} \pc{\{ e_{i} \}},
			\end{equation}
		we see that $K$ is a nonempty closed convex cone. Next, by \cref{eq: pc-ei},  \cite[Proposition 6.27]{bauschke2017convex}  implies that 
			\begin{math}
				K = \bigcap_{i\in I} \pc{\rb{\RR_{+}e_{i}}} = \pc{\rb*{\sum_{i\in I} \RR_{+} e_{i} }},
			\end{math}
		and since $\sum_{i \in I}\RR_{+} e_{i}$ is a nonempty  closed convex cone by \cref{eg: fi-cone}, taking the polar cones and invoking \cite[Corollary 6.34]{bauschke2017convex} yield
			\begin{math}
				\pc{K} = \rb*{\sum_{i\in I} \RR_{+} e_{i} }^{\ominus\ominus} = \sum_{i\in I} \RR_{+} e_{i} .
			\end{math}
		Hence, according to Moreau's theorem (\cref{fact: Moreau}) and \cref{thm: pr-K-ei}, we conclude that 
				\begin{math}
					P_{K}x = x - P_{\pc{K} }x = x - \sum_{i \in I} \max\left\{\scal{x}{e_{i}},0 \right\} e_{i}
				\end{math}
		and that 
			\begin{equation}
				d_{K}^{2}\rb{x} = \norm{x}^{2} - d_{\pc{K}}^{2}\rb{x} = \norm{x}^{2}- \rb*{\norm{x}^{2} - \sum_{i\in I} \rb{\max \left\{\scal{x}{e_{i}} ,0 \right\} }^{2} } = \sum_{i\in I} \rb{\max \left\{\scal{x}{e_{i}} ,0 \right\} }^{2},
			\end{equation}
		as claimed in \cref{eq: pr-Kei-po}.
	\end{proof}

\section{The projector onto the intersection of a cone and a
ball}

\label{s:main1}

Our first set of main results is presented in this section.
It turns out that the projector onto the intersection of a cone
and a ball has a pleasing explicit form.

\begin{theorem}[cone intersected with ball]
\label{thm: K-ba}
	Let $K$ be a nonempty closed convex cone in $\HH$, let $\rho \in \RR_{++}$, and set $C \coloneqq K \cap \ball{0}{\rho}$. Then 
			 \begin{equation}\label{eq: K-ba}
				\rb{\forall x \in \HH} \quad  P_{C}x = \frac{\rho}{\max \left\{\norm{P_{K}x}, \rho \right\} }  P_{K}x \quad  \text{and} \quad  d_{C}{\left(x\right)} = \sqrt{ d_{K}^{2}{\rb{x}} + \rb{\max \left\{\norm{P_{K}x} - \rho ,0\right\} }^{2}  }.
			\end{equation}
\end{theorem}

	\begin{proof}
		Take $x \in \HH$, set $\beta \coloneqq \rho / \max \left\{\norm{P_{K}x},\rho \right\} \in \RR_{++}$, and set $p \coloneqq \beta P_{K}x$. Then, since $K$ is a cone and $P_{K}x \in K$, we get  $p \in K$, and thus, since $\norm{p} = \beta \norm{P_{K}x} = \rho \rb{\norm{P_{K}x} / \max \{ \norm{P_{K}x},\rho  \}} \leq \rho $, it follows that  $p \in K \cap \ball{0}{\rho} = C$. Hence, because $C$ is closed and convex, in the light of \cref{eq: pr-C}, it remains to verify that  $\rb{\forall y \in C} ~ \scal{x-p}{y-p} \leq 0$. To this end,  take $y \in C$, and we  consider two alternatives: 
		
		\hspace{\parindent}(A) $\norm{P_{K}x} \leq \rho$: Then $\beta = \rho / \rho =1$. It follows that $p = P_{K}x$, and  so
		\begin{equation}\label{eq: dst-co1}
		\norm{x-p} = \norm{x-P_{K}x} = d_{K}{\left(x\right)}.
		\end{equation}
		Next, because $y \in K$, \cref{eq: pr-C} asserts that  $\scal{x-p}{y-p} = \scal{x-P_{K}x}{y-P_{K}x} \leq 0$. 
		
		\hspace{\parindent}(B) $\norm{P_{K}x} > \rho $: Then $\beta = \rho / \norm{P_{K}x} \in \left] 0 ,1 \right[$,  and so \cref{lem: cpt}\ref{it: cpt2} implies that 
		\begin{equation}\label{eq: dst-co2}
		\norm{x-p} = \sqrt{d_{K}^{2}{\left(x\right)} + \left(\norm{P_{K}x} - \rho \right)^{2} }.
		\end{equation}
		In turn, on the one hand, since $y$ belongs to the cone $K$, it follows  that $\left(1/\beta\right)y \in K$, from which and \cref{eq: pr-C} we deduce that 
		\begin{equation}\label{eq: co-b1}
		\scal{x - P_{K}x }{y - \beta P_{K}x}  = \beta \scal{x-P_{K}x}{\left(1/\beta\right)y - P_{K}x} \leq 0.
		\end{equation}
		On the other hand, because $y \in \ball{0}{\rho}$ and $\beta = \rho / \norm{P_{K}x}$,  the Cauchy{\textendash}Schwarz  inequality yields 
		\begin{equation}\label{eq: co-b2}
		\scal{P_{K} x}{y-\beta P_{K}x} = \scal{P_{K}x}{y} - \rho \norm{P_{K}x} \leq \norm{P_{K}x}\norm{y} - \rho \norm{P_{K}x} \leq 0. 
		\end{equation}
		Altogether, combining \cref{eq: co-b1}\&\cref{eq: co-b2} and using the  fact that  $\beta \in \left] 0,1 \right[$, we obtain 
		\begin{subequations} 
			\begin{align}
			\scal{x-p}{y-p} 
			& = \scal{x-\beta P_{K}x}{y-\beta P_{K}x} \\
			& = \scal{x-P_{K}x}{ y-\beta P_{K}x } + \left(1-\beta\right)\scal{P_{K}x}{y-\beta P_{K}x} \\
			& \leq 0.
			\end{align}
		\end{subequations}
		\indent Hence, in both cases, we have $\scal{x-p}{y-p} \leq 0$. Thus $ p =P_{C}x$, and it follows from \cref{eq: dst-co1}\&\cref{eq: dst-co2} that 
		\begin{equation}
		d_{C}\rb{x} = \norm{x- P_{C}x} = \norm{x-p} = \sqrt{d_{K}^{2}{\left(x\right)} +  \left( {\max}{ \left\{ \norm{P_{K}x} - \rho ,0 \right\}   } \right)^{2} },
		\end{equation}
		as stated in \cref{eq: K-ba}.
	\end{proof}

Here are some easy consequences of \cref{thm: K-ba}.

\begin{example}\label{eg: pr-ba}
	In the setting of \cref{thm: K-ba}, suppose that $K = \HH$. Then  $C = \ball{0}{\rho}$,  $P_{K} = \Id$,  $d_{K} \equiv 0$, and \cref{eq: K-ba} becomes  
	\begin{equation}\label{eq: pr-ba}
	\left(\forall x \in \HH\right)\quad P_{C}x = \frac{\rho}{\max \{ \norm{x},\rho \} }  x \quad \text{and} \quad d_{C}{\left(x\right)} =  \max \{ \norm{x} - \rho ,0 \}.
	\end{equation}
	We thus recover the formula for projectors onto balls.
\end{example}

\begin{corollary}\label{cor: pr-ba}
	Let $K$ be a nonempty closed convex cone in $\HH$, let
	$\rho \in \RR_{++}$, and set $C \coloneqq K \cap
	\ball{0}{\rho}$. Then\footnote{Here and
	elsewhere, ``$\circ$'' denotes the composition of
	operators.} $P_{C} = P_{\ball{0}{\rho}} \circ P_{K}$.
\end{corollary}
	
	\begin{proof}
		Combine \cref{eq: K-ba} and \cref{eq: pr-ba}.
		Alternatively, set\footnote{We use the symbol
		$\iota_{C}$ to denote the \emph{indicator
		function} of a subset $C$ of $\HH$:
		$\iota_C(x)=0$, if $x\in
		C$; $\iota_C(x)=+\infty$, if $x\notin C$.} $f\coloneqq  \iota_{\ball{0}{\rho}}$ and $\kappa \coloneqq  \iota_{K}$ in the equivalence (iii)\ensuremath{\Leftrightarrow}(iv) of \cite[Theorem 4]{yu2013decomposing}. (Note that $\iota_{\ball{0}{\rho}} + \iota_{K} = \iota_{C}$.)
	\end{proof}

\begin{remark}
	In the setting of \cref{cor: pr-ba},  as we shall see in \cref{eg: Kba-ncom},  $P_{C} \neq  P_{K}\circ P_{\ball{0}{\rho}}$, i.e., $P_{\ball{0}{\rho}}\circ P_{K} \neq P_{K}\circ P_{\ball{0}{\rho}}$,  in general.
\end{remark}

\begin{example}\label{eg: Kba-ncom}
	Suppose that $\HH = \RR^{2}$. Set $K \coloneqq
	\RR_{+}^{2}$ and  $x \coloneqq \rb{1,-1}$. Then (see also
	\cref{figgy}) 
		\begin{equation}
			\rb{P_{K}\circ P_{\ball{0}{1}}}x  = P_{K} \rb{ P_{\ball{0}{1}}x } \overset{\cref{eq: pr-ba}}{=} P_{K } \rb*{\frac{1}{\sqrt{2}},-\frac{1}{\sqrt{2}}} \overset{\cref{eq: pr-orthant}}{=} \rb*{\frac{1}{\sqrt{2}},0}
		\end{equation}
	and 
		\begin{equation}
			\rb{P_{\ball{0}{1}}\circ P_{K}}x = P_{\ball{0}{1}}\rb{P_{K}x } \overset{\cref{eq: pr-orthant}}{=} P_{\ball{0}{1}}\rb{1,0} \overset{\cref{eq: pr-ba}}{=} \rb{1,0}.
 		\end{equation}
 	Hence 
 		\begin{equation}
	 		P_{\ball{0}{1}}\circ P_{K} \neq P_{K}\circ P_{\ball{0}{1}}.
 		\end{equation}
\end{example}

	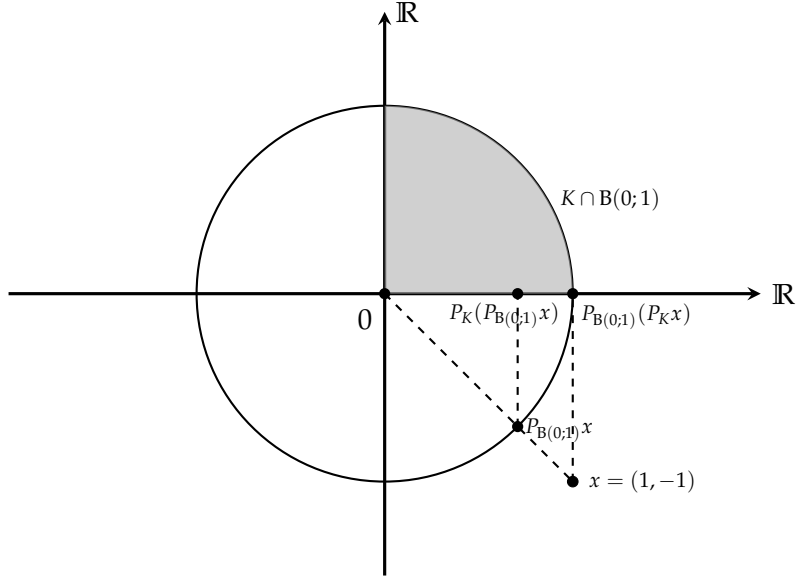
\begin{figure}[H]	
		\centering
		\begin{center}
			\begin{tikzpicture}[scale=2.5,
			axis/.style={-stealth,very thick},
			vector/.style={-latex,red, thick},
			vector guide/.style={dashed,red,thick},
			mynode/.style={fill,circle,inner sep=1.5pt,outer sep=0pt}
			]
			\draw[thick] (0,0) circle (1cm);
			\draw[axis] (-2,0)--(2,0) node [right] {$\RR$};
			\draw[axis] (0,-1.5)--(0,1.5) node [right] {$\RR$};
			
			\path[fill =gray!60, fill opacity=.6] (1,0) arc (0:90:1) -- (0,0) -- cycle;
			
			\coordinate (O) at (0,0);
			\node[mynode,fill=black,label={[label distance = 0.05mm]{-110}:{$0$} }] at (O) {};
			
			\coordinate (X) at (1,-1);
			\node[circle,draw=black,fill=black,inner sep=0pt, minimum size= 4pt,label={[label distance=0.05 pt]{0}: {\scriptsize  $x=(1,-1)$}}] at (X) {};

			\coordinate (A) at ( .7071067811865475244008445  , -.7071067811865475244008445  ) ;
			\node[mynode,fill=black] at (A) {};
			\node at ($(A) + (2.2mm,-0.2mm) $) {\scriptsize $P_{\ball{0}{1}}x $};

			\draw[black, thick, dashed] (O) -- (X);
			
			\coordinate (B) at (.7071067811865475244008445, 0);
			\node[mynode,fill=black ] at  (B)  {};
			\node at ($(B)+(-2pt,.25pt)$) [anchor = north]  { \scriptsize $P_{K}\rb{P_{\ball{0}{1}}x} $} ;
			\draw[black, thick, dashed] (A) -- (B) ;

			\node at (1.2,0.5) {\scriptsize$ K \cap \ball{0}{1}$} ;
			
			\coordinate (C) at (1,0);
			\node[mynode, fill =black, label= { [label distance = -1 mm]{-10} : {\scriptsize $P_{\ball{0}{1}}\rb{ P_{K}x } $ }  }] at (C) {};
			
			\draw[black, thick, dashed] (X) -- (C) ;

			\end{tikzpicture}
			\caption{\cref{eg: Kba-ncom} illustrates that the
			projectors onto a cone and ball may fail
			to commute. } 
		\label{figgy}
		\end{center}
	\end{figure}

As will be seen in the next result, \cref{eg: Kba-ncom} is, however, not a coincidence.

\begin{corollary}\label{cor: PK-Pba}
	Let $K$ be a nonempty closed convex cone in $\HH$, 
	and let $\rho \in \RR_{++}$. Then 
		\begin{equation}
			\rb{\forall x \in \HH} \quad \rb{P_{K} \circ P_{\ball{0}{\rho}}  }x = \frac{\rho}{ \max \left\{ \norm{x},\rho \right\} } P_{K}x.
		\end{equation}
\end{corollary}

	\begin{proof}
		It follows from \cref{eq: pr-ba} and \cite[Proposition 29.29]{bauschke2017convex}  that
			\begin{equation}
				\rb{\forall x \in \HH} \quad \rb{P_{K} \circ P_{\ball{0}{\rho}}  }x  
				 = P_{K} \rb{P_{\ball{0}{\rho } } x } \\ 
				 = P_{K}\rb*{ \frac{\rho}{\max\left\{ \norm{x},\rho \right\} } x }  \\ 
				 = \frac{\rho}{\max\left\{ \norm{x},\rho \right\} } P_{K}x, \\
			\end{equation} 
		as desired.
	\end{proof}

\begin{remark}
	Consider the setting of \cref{cor: PK-Pba}. Using \cref{cor: PK-Pba}, \cref{thm: K-ba}, and \cref{cor: pr-ba}, we deduce that 
		\begin{equation}
			\rb{\forall x \in \HH} \quad \rb{P_{K}\circ P_{\ball{0}{\rho}} }x = \frac{\max \{ \norm{P_{K}x },\rho \} }{\max\left\{ \norm{x},\rho \right\}} \rb{ P_{\ball{0}{\rho}} \circ P_{K} }x.
		\end{equation}
\end{remark}

\section{The projector onto the intersection of a cone and a
sphere}

In this section, which contains our second half of main results, 
we develop formulae for the projector onto the
intersection of a cone and a sphere. 

\label{s:main2}

\begin{theorem}\label{thm: K-S}
	Let $K$ be a nonempty closed convex cone in $\HH$, let $\rho \in \RR_{++}$, and set $C \coloneqq K \cap \sphere{0}{\rho}$. Suppose that $ K \neq \{0\}$. Then the following hold: 
		\begin{enumerate}
			\item\label{it: KS1} \begin{math}
				\rb{\forall x \in K^{\perp }}~ P_{C}x = C \text{~and~} d_{C}\rb{x} = \sqrt{\norm{x}^{2} + \rho^{2} }.
			\end{math}
			\item\label{it: KS2} \begin{math}
				\rb{\forall x \in \HH\smallsetminus \pc{K}} ~ P_{C}x = \left\{ \rb{\rho / \norm{ P_{K}x }  }  P_{K}x \right\} \text{~and~} d_{C}\rb{x} = \sqrt{ d_{K}^{2}\rb{x} + \rb{\norm{P_{K}x} - \rho }^{2}  }.
			\end{math}
		\end{enumerate}
\end{theorem}

	\begin{proof}
		We first observe that, by assumption and \cref{rem: clar}, $C \neq \varnothing$.
		
		\cref{it: KS1}: Fix $x \in K^{\perp}$. Then, for every $y \in C = K \cap \sphere{0}{\rho}$, since $x \perp y$ and $\norm{y} = \rho$, we get $\norm{x-y}^{2} = \norm{x}^{2} + \norm{y}^{2} = \norm{x}^{2} + \rho^{2}$.  It follows that $d_{C}\rb{x} = \sqrt{\norm{x}^{2} +\rho^{2} }$ and that  $P_{C}x = C$, as desired.
		
		\cref{it: KS2}: First, by the very definition of $C$, we see that 
			\begin{equation}\label{eq: KS-edst}
				\text{$C$ consists of vectors of equal norm.}
			\end{equation}
		Now take $x \in \HH \smallsetminus \pc{K}$; set\footnote{Due to \cref{lem: cpt}\cref{it: cpt1}, we have $P_{K}x \neq 0$.} $\alpha \coloneqq \rho / \norm{P_{K} x}  \in \RR_{++}$ and 
			\begin{equation}\label{eq: KS-ap}
				 \quad p \coloneqq \alpha P_{K}x.
			\end{equation}
		Then, because $P_{K}x$ belongs to the cone $K$, we obtain $p \in K$, and because  
			\begin{equation}\label{eq: KS-Np}
				\norm{p} = \norm*{ \frac{\rho}{\norm{P_{K}x } } P_{K}x } = \rho,
			\end{equation}
		it follows that 
			\begin{equation}\label{eq: KS-pC}
				p \in K \cap \sphere{0}{\rho} = C.
			\end{equation}
		Next, fix   $y \in C$. Since $y \in  C \subseteq
		K$ and $K$ is a cone, we have $\alpha^{-1}y \in
		K$. Therefore, since $\norm{y} = \rho$, we derive
		from \cref{eq: KS-ap}, \cref{eq: pr-C}, and \cref{eq: KS-Np}  that 
			\begin{subequations}
				\begin{align}
					\scal{x}{p} - \scal{x}{y}	
						& = \scal{x}{p-y} \\
						& = \scal{x- P_{K}x }{ p-y } + \scal{P_{K}x}{ p-y} \\
						& = \scal{x-P_{K}x }{\alpha P_{K}x - y} + \scal{\alpha^{-1}p }{ p - y} \\
						& = \alpha \underbrace{\scal{x - P_{K}x }{ P_{K}x - \alpha^{-1}y}}_{\geq 0 \text{~by \cref{eq: pr-C}}} + \alpha^{-1}\scal{ p}{p - y } \\
						& \geq \rb{2\alpha}^{-1}\rb{ \norm{ p }^{2} + \norm{p - y }^{2} - \norm{y}^{2} } \\
						& = \rb{2\alpha}^{-1}\rb{ \rho^{2} + \norm{p - y }^{2} - \rho^{2} } \\
						& = \rb{2\alpha}^{-1}\norm{p-y}^{2}.
				\end{align}
			\end{subequations}
		To summarize, we have shown that $\rb{\forall y \in C}~ y \neq p \Rightarrow \scal{x}{y} < \scal{x}{p}$. Combining this, \cref{eq: KS-pC}, and \cref{eq: KS-edst}, we infer from \cref{lem: pil1}\cref{it:p11} that $P_{C}x = \{p\}$. This and \cref{lem: cpt}\cref{it: cpt2} yield the latter assertion, and the proof is complete.
	\end{proof}

Let us provide some examples.

\begin{corollary}[Projections onto circles]\label{cor: pr-cir}
	Let $V$ be a nonzero closed linear subspace of $\HH$, let $\rho \in \RR_{++}$, and set $C \coloneqq V \cap \sphere{0}{\rho}$. Then 
		\begin{equation}\label{eq: pr-cir}
			\rb{\forall x \in \HH} \quad P_{C}x = \begin{cases} 
				C, & \text{if~} x \in V^{\perp}; \medskip \\
				\displaystyle \left\{ \frac{\rho}{\norm{P_{V}x}}P_{V}x \right\}, & \text{otherwise}.
			\end{cases}
		\end{equation}
\end{corollary}	
	
	\begin{proof}
		Combine  \cref{thm: K-S} and the fact that $V^{\ominus} = V^{\perp}$.
	\end{proof}

\begin{remark}
	Letting $V = \HH$ in \cref{cor: pr-cir}, we see that $C = \sphere{0}{\rho}$, that $V^{\perp} = \{0\}$, that $P_{V} = \Id$, and that \cref{eq: pr-cir} becomes 
		\begin{equation}
			\rb{\forall x \in \HH} \quad P_{C}x = \begin{cases}
				C, & \text{if~} x = 0;  \medskip \\
				\displaystyle \left\{ \frac{\rho}{\norm{x}}x \right\}, & \text{otherwise}.
			\end{cases}
		\end{equation}
	Hence, we recover the well-known formula for projectors onto spheres.
\end{remark}

\begin{example}\label{eg: pr-Cal}
 Let $\alpha \in \RR$ and $\beta \in \RR_{++}$, and set 
		\begin{equation}
			\b{S}_{\alpha,\beta} \coloneqq \sphere{0}{\beta} \times \{\alpha \}. 
		\end{equation}
	Then 
		\begin{equation}\label{eq: pr-Cal}
			\rb{\forall \bx = \rb{x,\xi} \in \b{\HH}}\quad P_{\b{S}_{\alpha,\beta}}\bx = \begin{cases}
			\b{S}_{\alpha,\beta}, & \text{if~} x =0; \medskip \\
			\displaystyle \left\{ \rb*{\frac{\beta}{\norm{x} }x,\alpha  } \right\}, & \text{otherwise.}
			\end{cases}
		\end{equation}
\end{example}

	\begin{proof}
		Set $\boldsymbol{V} \coloneqq \HH \times \{0\}$, which is a nonzero closed linear subspace of $\b{\HH}$ by \cref{H}. Let us first observe that 
			\begin{equation}\label{eq: hp-sum}
				\b{V} = \menge{\bx = \rb{x,\xi} \in \b{\HH} }{ \scal{\bx}{\rb{0,1} } = 0} = \{\rb{0,1}\}^{\perp},
			\end{equation}
		and thus,
			\begin{equation}\label{eq: Vperp}
				\rb{\forall \bx = \rb{x,\xi} \in \boldsymbol{\HH}} \quad \bx \in \boldsymbol{V}^{\perp} \iff \bx \in \RR \rb{0,1} \iff x=0.  
			\end{equation}
		Moreover, it is straightforward to verify that 
			\begin{equation}\label{eq: C0}
				\b{S}_{0,\beta} = \boldsymbol{V}\cap \dsphere{0}{\beta}.
			\end{equation}
		Now fix $\bx = \rb{x,\xi} \in \boldsymbol{\HH}$. Then,  appealing to \cite[Example 3.23]{bauschke2017convex} and \cref{eq: hp-sum}, we see that 
			\begin{math}
				P_{ \b{V} }\bx = \rb{x,0},
			\end{math}
		Combining this, \cref{eq: C0}, and \cref{eq: Vperp}, we deduce from \cref{cor: pr-cir} that 
		\begin{subequations} 
			\begin{align}
				P_{\b{S}_{ 0,\beta}}\bx = P_{\b{S}_{ 0,\beta}} \rb{x,\xi} 
				& = \begin{cases}
					\b{S}_{0,\beta}, & \text{if~} x =0; \medskip \\
					\displaystyle \left\{ \frac{\beta}{\norm{P_{\b{V}}\bx }} P_{\b{V}}\bx\right\}, & \text{otherwise}
				\end{cases} \\
				& = \begin{cases}
				\b{S}_{0,\beta}, & \text{if~} x =0; \medskip \\
				\displaystyle \left\{ \rb*{\frac{\beta}{\norm{x} }x,0 }\right\}, & \text{otherwise.}
				\end{cases} \label{eq: pr-C0}
			\end{align}
		\end{subequations}
	Consequently, since\footnote{As the reader can easily verify.} $\b{S}_{\alpha,\beta} = \rb{0,\alpha} + \b{S}_{0,\beta} $, we derive from \cref{eq: pr-C0} (applied to the point $\rb{x,\xi - \alpha}$)  that 
		\begin{subequations}
			\begin{align}
				P_{\b{S}_{\alpha,\beta}}\bx 
					& = \rb{0,\alpha} +P_{\b{S}_{0,\beta}}\rb{ \bx - \rb{0,\alpha} }  \\
					& = \rb{0,\alpha} +P_{\b{S}_{0,\beta}} \rb{x,\xi - \alpha}  \\
					& =\begin{cases}
						\rb{0,\alpha} + \b{S}_{0,\beta}, & \text{if~} x =0; \medskip \\
					\displaystyle \rb{0,\alpha} + \left\{ \rb*{\frac{\beta}{\norm{x} }x,0 }\right\}, & \text{otherwise}
					\end{cases} \\
					& = \begin{cases}
						 \b{S}_{\alpha,\beta}, & \text{if~} x =0; \medskip \\
					\displaystyle \left\{ \rb*{\frac{\beta}{\norm{x} }x,\alpha  }\right\}, & \text{otherwise,}
					\end{cases}
			\end{align}
		\end{subequations}
	as announced in \cref{eq: pr-Cal}.
	\end{proof}

Next, we turn to the more complicated case when the point to be
projected belongs to the polar cone.

	\begin{theorem}\label{thm: pK}
		Let $K$ be a  convex cone in $\HH$ such that $K \smallsetminus \{0\} \neq \varnothing$, let  $\rho$ be in $\RR_{++}$, and let $x \in \pc{K}$. Suppose that  there exists a nonempty subset $C$ of $K$ such that   
			\begin{equation}\label{eq: pK-eC}
				\left(\forall y \in C\right) \quad  \norm{y} = \rho
			\end{equation}
		and that 
			\begin{equation}\label{eq: genK}
				K = \pos{C}.
			\end{equation} 
		Set  
		\begin{equation}
			D \coloneqq  K \cap \sphere{0}{\rho} \quad \text{and} \quad \kappa \coloneqq \sup \scal{x}{C}.
		\end{equation}
		Then the following hold: 
			\begin{enumerate}
				\item\label{it: pK-1} Suppose that $P_{C}x = \varnothing$. Then $P_{D}x =\varnothing$.
				\item Suppose that $P_{C}x \neq \varnothing$, and set $E \coloneqq \sphere{0}{\rho} \cap \cone\rb{ \conv{P_{C}x} }$. Then the following hold: 
					\begin{enumerate}
						\item\label{it: pK21} $P_{C}x \subseteq P_{D}x \subseteq E$ and $\max \scal{x}{D} = \max \scal{x}{C}$.
						\item\label{it: pK22} Suppose that $\kappa < 0$. Then $P_{D}x = P_{C}x $.
						\item\label{it: pK23} Suppose that $\kappa =0$. Then $P_{D}x = E$.
					\end{enumerate}
				\item\label{it: pK-3} $P_{C}x \neq \varnothing \Leftrightarrow P_{D}x \neq \varnothing.$
			\end{enumerate}
	\end{theorem}

\begin{proof}
	We start with a few observations. First, since $K\neq \{0\}$ by assumption, it follows from \cref{rem: clar} that  $D \neq \varnothing$. Next, in view of \cref{eq: pK-eC} and the assumption that $C\subseteq K$, we have  
		\begin{equation}\label{eq: pK-kaD}
			C\subseteq D.
		\end{equation}
	In turn, because $x \in \pc{K}$, we get from \cref{eq: genK} and \cref{lem: cone}\cref{it: c01} that 
			\begin{equation}\label{eq: pK-ka0}
				\kappa 
					 \leq 0.
			\end{equation}
	Finally, by the very definition of $D$, we see that 
		\begin{equation}\label{eq: pK-nrmD}
			\text{the vectors in $D$ are of equal norm.}
		\end{equation}
	
	\cref{it: pK-1}: We  prove the contrapositive and therefore assume that there exists 
		\begin{equation}\label{eq: pK-ctp}
			u \in P_{D}x.		
		\end{equation} 
	Then, by \cref{eq: pK-kaD}, \cref{eq: pK-nrmD}, \cref{eq:
	pK-ctp}, and  \cref{lem: pil1}\cref{it:p11}, we obtain  
		\begin{equation}\label{eq: pK-kaxu}
			\kappa  = \sup \scal{x}{C} \leq \sup \scal{x}{D} = \scal{x}{u}.
		\end{equation}
     In turn, combining  \cref{eq: pK-eC}, \cref{eq: pK-ka0}, 
     \cref{eq: pK-kaxu}, and the fact that 
		\begin{math}
			u \in D  = \rb{\pos{C}} \cap \sphere{0}{\rho},
		\end{math}
	we infer from  \cref{lem: pil2}\cref{it: p22a} that $P_{C}x \neq \varnothing$. 
	
	\cref{it: pK21}: Let us first prove that $P_{C}x \subseteq P_{D}x$ and that $\max \scal{x}{D} = \max \scal{x}{C}$. To this end, take $u \in P_{C}x$ and $y \in D$. Then,  because $y \in D \subseteq \pos{C}$, there exist finite sets $\set{\alpha_{i}}{ i\in I } \subseteq \RR_{+}$  and $\set{x_{i}}{i\in I} \subseteq C$ such that $y = \sum_{i\in I} \alpha_{i} x_{i}$. In turn, on the one hand, since $\norm{y} = \rho$, we infer from \cref{eq: pK-eC} and \cref{lem: pil2}\cref{it: p21} that $\sum_{i \in I}\alpha_{i} \geq 1$. On the other hand, since $u \in P_{C}x$, it follows from \cref{eq: pK-eC} and \cref{lem: pil1}\cref{it:p11} that  
		\begin{equation}\label{eq: xu-maxC}
			\scal{x}{u} =\max \scal{x}{C} = \kappa .
		\end{equation}
	So altogether, since $\rb{\forall i \in I}~x_{i} \in C$, using \cref{eq: pK-ka0}, we see that 
		\begin{equation}\label{eq: xu-maxD}
			\scal{x}{y} = \sum_{i \in I} \alpha_{i} \scal{x}{x_{i}} \leq \sum_{i \in I} \alpha_{i} \kappa = \kappa \sum_{i\in I}\alpha_{i} \leq \kappa = \scal{x}{u}. 
		\end{equation}
	Therefore, since  $u \in C \subseteq D$ by \cref{eq: pK-kaD}, we derive from \cref{eq: xu-maxD} and \cref{eq: xu-maxC} that   
			\begin{equation}\label{eq: xuMaD}
				\max \scal{x}{D}
				  = \scal{x}{u}  \\
				  = \max \scal{x}{C}. 
			\end{equation}
	Also, appealing to \cref{eq: xuMaD} and \cref{eq: pK-nrmD}, we get from \cref{lem: pil1}\cref{it:p11} that $u \in P_{D}x$, as desired. It now remains to establish the inclusion $P_{D}x \subseteq E$. To do so, fix $v \in P_{D}x$. Then, in view of \cref{eq: pK-nrmD}, \cref{lem: pil1}\cref{it:p11} and \cref{eq: xuMaD}\&\cref{eq: xu-maxC}\&\cref{eq: pK-ka0} assert that  
		\begin{equation}\label{eq: pK-v1}
			\scal{x}{v} = \max \scal{x}{D} = \max \scal{x}{C} \leq 0.
		\end{equation}
	Thus, since  
		\begin{equation}\label{eq: pK-v2}
			v \in D = \rb*{\pos{C} } \cap \sphere{0}{\rho},
		\end{equation}
	it follows from \cref{eq: pK-eC} and \cref{lem: pil2}\cref{it: p22b} that $v \in \sphere{0}{\rho} \cap \cone \rb{\conv{P_{C}x } } = E$, as claimed.
	
	\cref{it: pK22}: Consider the element $v \in P_{D}x$ of the proof of \cref{it: pK21}. Combining \cref{eq: pK-eC}\&\cref{eq: pK-v1}\&\cref{eq: pK-v2} and the assumption that $\kappa < 0$, we derive from \cref{lem: pil2}\cref{it: p22c} that $v \in P_{C}x$, and hence, $P_{D}x \subseteq P_{C}x$. Consequently, since $P_{C}x \subseteq P_{D}x$ by \cref{it: pK21}, the assertion follows. 
	
	\cref{it: pK23}: According to \cref{it: pK21}, it suffices to show that $E\subseteq P_{D}x$. Towards this end, take $ w \in E$ and $y \in D$. By the very definition of $E$, there exist
	finite sets $\set{\beta_{j}}{j \in J} \subseteq \RR_{++}$ and $\set{x_{j}}{j \in J} \subseteq P_{C}x$ such that $w =  \sum_{j \in J} \beta_{j} x_{j}$. In turn, since $\set{x_{j}}{j \in J} \subseteq P_{C}x$, we get from \cref{eq: pK-eC} and \cref{lem: pil1}\cref{it:p11} that  $\rb{\forall j \in J}~ \scal{x}{x_{j}} = \kappa = 0$, from which and \cref{eq: xuMaD} it follows that 
		\begin{equation}
			\scal{x}{w} =  \sum_{j \in J} \beta_{j} \scal{x}{x_{j}} =0 = \kappa = \max \scal{x}{D}.
		\end{equation}
	Consequently, since $w \in E \subseteq D$ by the very definitions of $E$ and $D$, invoking  \cref{eq: pK-nrmD} and \cref{lem: pil1}\cref{it:p11} once more, we conclude that $w \in P_{D}x$, as required.
	
	\cref{it: pK-3}: Combine \cref{it: pK-1} and \cref{it: pK21}.
\end{proof}

We are now ready for the main result of this section which
provides a formula for the projector of a finitely generated cone
and a sphere.

\begin{corollary}[cone intersected with sphere]
\label{cor: fi-KS}
	Let $\set{x_{i}}{i \in I}$ be a nonempty finite subset of $\HH$, let $\rho \in \RR_{++}$, and let $x \in \HH$. Set 
		\begin{equation}
			K \coloneqq \sum_{i \in I} \RR_{+} x_{i}, \quad C \coloneqq K \cap \sphere{0}{\rho}, \quad \kappa \coloneqq \max_{i \in I} \scal{x}{x_{i}}, \quad \text{and~} I\rb{x} \coloneqq \menge{ i \in I }{ \scal{x}{x_{i}} = \kappa }.
		\end{equation}
	Suppose that $\rb{\forall i \in I}~ \norm{x_{i}} = \rho$. Then 
		\begin{equation}
			P_{C}x = \begin{cases}
				\displaystyle \left\{ \frac{\rho}{ \norm{P_{K}x } } P_{K}x \right\}, & \text{if~} \kappa >0; \medskip \\
				\sphere{0}{\rho} \cap \cone \rb*{\conv \set{x_{i}}{i\in I\rb{x}  }}, & \text{if~} \kappa =0; \medskip \\
				\set{x_{i}}{i\in I\rb{x} }, & \text{if~} \kappa <0.
			\end{cases}
		\end{equation}
\end{corollary}
	
	\begin{proof}
		Set $X \coloneqq \set{x_{i}}{i \in I}$. First, it follows from \cref{eg: fi-cone} that $K$ is a nonempty closed convex cone. In addition,  \cref{lem: cone}\cref{it: c01} (applied to $\set{x_{i}}{ i\in I} $) implies that 
			\begin{equation}\label{eq: ka-polar}
				x \in \pc{K}  \iff \kappa = \max_{i \in I} \scal{x}{x_{i}} \leq 0.
			\end{equation}
		Next, due to our assumption,  \cref{lem: pil1}\cref{it:p11} yields  
			\begin{equation}\label{eq: PX}
				P_{X}x = \set{x_{i}}{ i \in I\rb{x} } \neq \varnothing.
			\end{equation}
		Let us now identify $P_{C}x$ in each of the following conceivable cases: 
			
			\hspace{\parindent}(A) $\kappa >0$: Then, by \cref{eq: ka-polar}, we have $x \in \HH\smallsetminus \pc{K}$, and hence, \cref{thm: K-S}\cref{it: KS2} asserts that $P_{C}x = \left\{ \rb{\rho / \norm{ P_{K}x }  }  P_{K}x \right\}$.
			
			\hspace{\parindent}(B) $\kappa =0$: Using \cref{thm: pK}\cref{it: pK23} (with the set $C$  being $X=\set{x_{i}}{i\in I}$) and \cref{eq: PX},  we obtain 
				\begin{math}
					P_{C}x = \sphere{0}{\rho}  \cap \cone \rb{\conv \set{x_{i}}{i\in I \rb{x} }  }.
				\end{math}
				
			\hspace{\parindent}(C) $\kappa < 0$: Invoking \cref{thm: pK}\cref{it: pK22} and \cref{eq: PX}, we immediately have $P_{C}x = \set{x_{i} }{i \in I \rb{x} }$.
	\end{proof}

\begin{remark}
	Consider the setting of \cref{cor: fi-KS}. Since 
		\begin{math}
			\set{x_{i}}{i\in I \rb{x} } \subseteq \sphere{0}{\rho} \cap \cone \rb{\conv \set{x_{i}}{i\in I\rb{x} }  }
		\end{math}
 	 by the assumption that $\norm{x_{i}} \equiv \rho$, we see that 
		\begin{equation}
			\operatorname{s} \colon \HH \to \HH : x \mapsto \begin{cases}
				\displaystyle \frac{\rho}{\norm{P_{K}x }} P_{K}x, & \text{if~} \displaystyle  \max_{ i\in I} \scal{x}{x_{i}} >0; \medskip \\
				\operatorname{s}\rb{x} \in \set{ x_{i} }{ i \in I \rb{x} }, & \text{otherwise}
			\end{cases}
		\end{equation}
	is a selection of $P_{C}$.
\end{remark}

\begin{example}\label{eg: eg-ortho}
	Consider the setting of \cref{thm: pr-K-ei}. Set
	 \begin{equation}
		 C\coloneqq K \cap \sphere{0}{ 1 }, ~ \kappa \coloneqq \max_{i \in I} \scal{x}{e_{i}}, ~ I\rb{x} \coloneqq \menge{ i \in I }{ \scal{x}{e_{i}} = \kappa }, ~ \text{and~} \lambda \coloneqq \sqrt{ \sum_{i\in I} \rb{\max \left\{\scal{x}{e_{i}} ,0 \right\} }^{2}}.
	 \end{equation}
	 Then 
	 	\begin{equation}\label{eq: eg-ei}
		 	P_{C}x = \begin{cases}
			 	\displaystyle \left\{ \lambda^{-1} \sum_{i \in I}\max \left\{\scal{x}{e_{i}},0\right\}  e_{i}  \right\}, & \text{if~} \kappa >0; \medskip \\
			 	\displaystyle \menge*{ \sum_{i \in I\rb{x} } \alpha_{i} e_{i} }{ \set{\alpha_{i}  }{i \in I \rb{x} }\subseteq \RR_{+} \text{~such that~} \sum_{i \in I\rb{x}}\alpha_{i}^{2} =1 }, & \text{if~} \kappa =0; \medskip \\
			 	\set{e_{i}}{ i \in I\rb{x}}, &\text{if~} \kappa < 0.
		 	\end{cases}
	 	\end{equation}
\end{example}
	
	\begin{proof}
		Since \begin{math}
				P_{K}x = \sum_{i \in I} \max\left\{\scal{x}{e_{i}},0 \right\} e_{i}
		\end{math}
	by \cref{eq: pr-ei}, we obtain 
		\begin{equation}\label{eq: eg-nPK}
			\norm{ P_{K}x }^{2} = \norm*{ \sum_{i \in I} \max\left\{\scal{x}{e_{i}},0 \right\} e_{i}}^{2} = \sum_{i\in I} \rb{\max \left\{\scal{x}{e_{i}} ,0 \right\} }^{2} = \lambda^{2} .
		\end{equation}
	 Next, let us show that 
		\begin{equation}\label{eq: eg-do}
			\sphere{0}{1} \cap \cone \rb*{\conv \set{e_{i}}{i\in I\rb{x}} } = \menge*{ \sum_{i \in I\rb{x}} \alpha_{i} e_{i} }{ \set{\alpha_{i}  }{i \in I\rb{x}} \subseteq \RR_{+} \text{~such that~} \sum_{i \in I\rb{x}}\alpha_{i}^{2} =1 }.
		\end{equation}
	To this end, denote the set on the right-hand side of \cref{eq: eg-do} by $D$. Take $y \in \sphere{0}{1} \cap \cone \rb{\conv \set{e_{i}}{i\in I\rb{x}} }$. Then there exist $\lambda \in \RR_{++}$ and $\set{ \alpha_{i} }{ i\in I\rb{x} } \subseteq \RR_{+}$ such that $y = \lambda \sum_{i \in I\rb{x}} \alpha_{i} e_{i}  = \sum_{i \in I\rb{x} } \rb{\lambda \alpha_{i}}e_{i} $. Furthermore, since $\set{e_{i}}{i \in I\rb{x} }$ is an orthonormal set, we get 
		\begin{math}
			1 = \norm{y}^{2} = \norm{ \sum_{i \in I\rb{x}}\rb{\lambda \alpha_{i}} e_{i} }^{2} = \sum_{i \in I\rb{x}} \rb{\lambda \alpha_{i}}^{2}.
		\end{math}
	Hence $y \in D$. Conversely, fix $z \in D$, say $z = \sum_{i \in I\rb{x}} \beta_{i} e_{i}$, where $\set{\beta_{i} }{i \in I\rb{x}} \subseteq \RR_{+}$ satisfying $\sum_{i \in I\rb{x}}  \beta_{i}^{2} =1 $, and set $\beta \coloneqq  \sum_{i \in I\rb{x}} \beta_{i}$. It is clear that $\beta >0$, and therefore, 
		\begin{math}
			z = \beta \sum_{i \in I\rb{x}} \rb{\beta_{i} / \beta }e_{i} \in \cone \rb{\conv \set{e_{i}}{i\in I\rb{x}} } .
		\end{math}
	In turn, because $\norm{z}^{2} = \sum_{i \in I\rb{x}} \beta_{i}^{2} =1$, it follows that $z \in \sphere{0}{1} \cap \cone \rb{\conv \set{e_{i}}{i\in I\rb{x}} }$. Thus \cref{eq: eg-do} holds. Consequently, using \cref{eq: pr-ei}\&\cref{eq: eg-nPK}\&\cref{eq: eg-do}, we obtain \cref{eq: eg-ei} via \cref{cor: fi-KS}. 
	\end{proof}

The following nice result was mentioned in \cite[Example 5.5.2 and
Problem 5.6.14]{lange2016mm}.

\begin{example}[Lange] \label{eg: Lange}
	Suppose that $\HH= \RR^{N}$, that $I = \left\{ 1,\ldots,N\right\}$,
	and that $\fa{ e_{i}}{ i\in I} $ is the canonical orthonormal
	basis of $\HH$.  Set 
\begin{equation}
K \coloneqq \RR_{+}^{N}\quad\text{and}\quad
C \coloneqq K \cap \sphere{0}{1}. 
\end{equation}
Now let $x =
	\fa{\xi_{i}}{i \in I} \in \HH$; set  $\kappa \coloneqq
	\max_{i \in I}\xi_{i}$,   $I \rb{x} \coloneqq \menge{ i \in
	I }{ \xi_{i} =\kappa }$, and  $x_{+}\coloneqq \fa{\max
	\left\{\xi_{i},0 \right\} }{ i \in I }$. Then
		\begin{equation}\label{eq: eg-RN}
		P_{C}x = \begin{cases}
		\displaystyle \left\{ \frac{1}{\norm{x_{+} } }x_{+} \right\}, & \text{if~} \kappa >0; \medskip \\
		\displaystyle \menge*{ \sum_{i \in I\rb{x}} \alpha_{i} e_{i} }{ \set{\alpha_{i}  }{i \in I\rb{x}} \subseteq \RR_{+} \text{~such that~} \sum_{i \in I\rb{x}}\alpha_{i}^{2} =1 }, & \text{if~} \kappa =0; \medskip \\
		\set{e_{i}}{ i \in I\rb{x}}, &\text{if~} \kappa < 0.
		\end{cases}
		\end{equation}
\end{example}
	
	\begin{proof}
		Because $\rb{\forall i \in I}~ \scal{x}{e_{i}} = \xi_{i}$ and $\norm{x_{+}}^{2} = \sum_{i \in I}\rb{\max \left\{\xi_{i},0 \right\}}^{2} $, \cref{eq: eg-RN} therefore follows from \cref{eg: eg-ortho}.
	\end{proof}

\section{Further examples}

\label{s:fex}

In this section, we provide further examples based on the Lorentz
cone and on the cone of positive semidefinite matrices.

\begin{example}
	Let $\alpha$ and $\rho$ be in $\RR_{++}$, let  
		\begin{equation}\label{eq: Lor-cone}
			\bK_{\alpha} = \menge*{ \rb{x,\xi} \in \HH \oplus \RR }{ \norm{x} \leq \alpha \xi } 
		\end{equation}
	be the Lorentz cone of parameter $\alpha$ of \cref{eg: ic-cone},  set $\b{C} \coloneqq \b{K}_{\alpha}\cap \dsphere{0}{\rho}$, and let $\bx =\rb{x,\xi} \in \b{\HH}$. Then 
		\begin{equation}\label{eq: pr-LC}
			P_{\bC}\bx = \begin{cases}
			\displaystyle \left\{\frac{\rho }{\norm{\bx}}\bx \right\}, & \text{if~} \norm{x} \leq \alpha \xi \text{~and~} \xi >0; \medskip  \\
			\displaystyle \left\{ \frac{\rho }{\sqrt{1+\alpha^{2}}} \rb*{\frac{\alpha x}{\norm{x} },1 } \right\}, & \text{if~} \norm{x} > \max \{ \alpha\xi , -\xi/\alpha \} \text{~or~} \left[\, x\neq 0 \text{~and~} \norm{x} \leq -\xi /\alpha \,\right]; \medskip \\
			\displaystyle \sphere{0}{\beta} \times \{\beta/\alpha\}, & \text{if~} x =0 \text{~and~} \xi < 0; \medskip \\
			\bC,& \text{if~} \rb{x,\xi } = \rb{0,0}.				
			\end{cases}
		\end{equation}
\end{example}

	\begin{proof}
		 Set 
		 	\begin{equation}\label{eq: ma-beta}
			 	\beta \coloneqq \frac{\rho \alpha }{ \rb{1+\alpha^{2}}^{1/2} } \in \RR_{++},
		 	\end{equation}
		 $\b{C}_{\alpha,\beta} \coloneqq \sphere{0}{\beta} \times \{\beta/\alpha\},\text{~and~} \kappa \coloneqq \max \scal{\bx}{\b{C}_{\alpha,\beta} }$. Then  it is readily verified that 
			\begin{equation}\label{eq: nrm-Cab}
				\rb{\forall \by \in \b{C}_{\alpha,\beta} } \quad \norm{\by} = \rho, 
			\end{equation}
		and due to \cref{lem: maxSph},  
			\begin{equation}\label{eq: ma-kappa}
				\kappa = \beta \norm{x} + \xi \beta / \alpha.
			\end{equation}
		Furthermore,  by \cref{eg: ic-cone}, 
			\begin{equation}\label{eq: g-Ka}
				\bK_{\alpha} =\pos{\bC_{\alpha,\beta}} = \cone\rb{\conv {\bC_{\alpha,\beta}} } \cup \{\bzero \},
			\end{equation}
		and by \cref{eg: pr-Cal} (applied to $\b{C}_{\alpha,\beta}$), we have 
			\begin{equation}\label{eq: PCab0} 
				\varnothing  \neq  P_{\b{C}_{\alpha,\beta}}\bx   
				 = \begin{cases}
				\b{C}_{\alpha,\beta}, & \text{if~} x =0; \medskip \\
				\displaystyle \left\{ \rb*{\frac{\beta}{\norm{x} }x, \frac{\beta}{\alpha}  }\right\}, & \text{otherwise.}
				\end{cases} 
			\end{equation}
		Let us now identify $P_{\bC}\bx$ in the following conceivable cases: 
			
			\hspace{\parindent}(A) $\norm{x} > -\xi /\alpha$: Then $\kappa > 0$ by \cref{eq: ma-kappa}, and so by \cref{eq: g-Ka} and \cref{lem: cone}\cref{it: c01}, $\bx \in \b{\HH} \smallsetminus \b{K}_{\alpha}^{\ominus}$. In turn, it follows from \cref{thm: K-S}\cref{it: KS2} (applied to $\bC = \b{K}_{\alpha}\cap \dsphere{0}{\rho}$) that  
				\begin{equation}\label{eq: ic-PC}
					P_{\b{C} }\bx = \left\{  \frac{\rho}{\norm{ P_{ \b{K}_{\alpha}} \bx } }   P_{\b{K}_{\alpha}}\bx \right\}.
				\end{equation}
			To evaluate $P_{\b{C}}\bx$ further, we consider two subcases: 
				
				\hspace{\parindent}\hspace{\parindent}(A.1) $\norm{x} \leq \alpha \xi$: Then $\bx \in \b{K}_{\alpha}$ by \cref{eq: Lor-cone}, and so  $P_{\b{K}_{\alpha}}\bx = \bx$, which yields $P_{\b{C} }\bx = \left\{ \rb{\rho / \norm{\bx}} \bx \right\} $.
				
				\hspace{\parindent}\hspace{\parindent}(A.2) $\norm{x} > \alpha \xi$: Then, according to \cite[Exercise 29.11]{bauschke2017convex}, 
					\begin{equation}\label{eq: Lc-A2a}
						P_{\bK_{\alpha} }\bx = P_{\bK_{\alpha} }\rb{x,\xi} = \frac{\alpha\norm{x}+\xi  }{1+\alpha^{2}}\rb*{\frac{\alpha x}{\norm{x}},1},
					\end{equation}
				and since $\alpha \norm{x} + \xi >0$, it follows that 
					\begin{equation}\label{eq: Lc-A2b}
						\norm{ P_{ \bK_{\alpha} } \bx } = \frac{\alpha\norm{x}+\xi  }{1+\alpha^{2}} \norm*{ \rb*{\frac{\alpha x}{\norm{x}},1} }= \frac{\alpha\norm{x}+\xi  }{1+\alpha^{2}} \sqrt{ \norm*{\frac{\alpha x}{\norm{x} }}^{2} +1 } = \frac{\alpha \norm{x} +\xi }{ \sqrt{1+\alpha^{2} } }.
					\end{equation}
				Hence, combining \cref{eq: ic-PC}\&\cref{eq: Lc-A2a}\&\cref{eq: Lc-A2b}, we get 
					\begin{equation}
						P_{\b{C}}\bx = \left\{\frac{\rho}{\sqrt{1+\alpha^{2}}} \rb*{ \frac{\alpha x}{\norm{x}},1 } \right\}.
					\end{equation}
			
			\hspace{\parindent}(B) $\norm{x} = -\xi /\alpha$: Then $\kappa = 0$ by \cref{eq: ma-kappa}, and  invoking \cref{eq: nrm-Cab}\&\cref{eq: g-Ka}\&\cref{eq: PCab0}, \cref{thm: pK}\cref{it: pK23} asserts that 
				\begin{equation}\label{eq: Lc-BPC}
					P_{\b{C}}\bx = \dsphere{0}{\rho} \cap \cone\rb{ \conv{ P_{\b{C}_{\alpha,\beta} } \bx } }.
				\end{equation}
			We consider two subcases: 
				
				\hspace{\parindent}\hspace{\parindent}(B.1) $x =0$: Then $\xi = 0$ and so $\bx = \rb{x,\xi} = 0$. Moreover, due to \cref{eq: PCab0},  $P_{\b{C}_{\alpha,\beta} }\bx  = \b{C}_{\alpha,\beta} $. Therefore, by \cref{eq: g-Ka} and \cref{eq: Lc-BPC},  
				\begin{subequations}
					\begin{align}
						\bC & = \bK_{\alpha} \cap \dsphere{0}{\rho}  \\
						& = \rb{\cone\rb{\conv {\bC_{\alpha,\beta}} } \cup \{\bzero \}} \cap \dsphere{0}{\rho} \\
						& = \cone\rb{\conv {\bC_{\alpha,\beta}} } \cap \dsphere{0}{\rho} \\ 
						& = \cone\rb{\conv { P_{\b{C}_{\alpha,\beta} }\bx   }} \cap \dsphere{0}{\rho} \\
						& = P_{\bC}\bx.
					\end{align}
				\end{subequations}
			
				\hspace{\parindent}\hspace{\parindent}(B.2) $x  \neq 0$: Then  \cref{eq: PCab0} yields $P_{\bC_{\alpha,\beta}} \bx = \left\{ \rb{ \beta x /\norm{x}, \beta / \alpha }\right\}$. In turn, since  $\norm{ \rb{\beta x /\norm{x}, \beta /\alpha } } =\rho$ by \cref{eq: ma-beta} and a simple computation, we obtain from \cref{eq: Lc-BPC} and  \cref{fact: cone}\cref{it: cone1} that  
				\begin{subequations}
					\begin{align}
						P_{\b{C}}\bx 
							& = \dsphere{0}{\rho} \cap \cone\rb{ \conv{ P_{\b{C}_{\alpha,\beta} } \bx } } \\
							& = \dsphere{0}{\rho} \cap \rb*{ \RR_{++} \rb*{ \frac{\beta x}{\norm{x}}, \frac{\beta}{\alpha} }  }  \\
							& =  \left\{  \rb*{ \frac{ \beta x}{\norm{x}} , \frac{\beta}{ \alpha }} \right\} \\
							& = \left\{  \frac{\beta}{\alpha} \rb*{ \frac{\alpha x}{\norm{x}},1 } \right\} \\
							& = \left\{ \frac{\rho }{\sqrt{1+\alpha^{2}}} \rb*{\frac{\alpha x}{\norm{x} },1 } \right\}.
					\end{align}
				\end{subequations} 
			
			\hspace{\parindent}(C) $\norm{x} < -\xi /\alpha$: Then $\kappa < 0$ by \cref{eq: ma-kappa}, and so, in view of \cref{eq: nrm-Cab}\&\cref{eq: g-Ka}\&\cref{eq: PCab0}, we deduce from \cref{thm: pK}\cref{it: pK22} that  $P_{\bC}\bx = P_{\bC_{\alpha,\beta} } \bx $.  Hence, by \cref{eq: PCab0} and \cref{eq: ma-beta}, we get 
				\begin{subequations}
					\begin{align}
						P_{\bC}\bx   & = \begin{cases}
						\b{C}_{\alpha,\beta}, & \text{if~} x =0; \medskip \\
						\displaystyle \left\{ \rb*{\frac{\beta}{\norm{x} }x, \frac{\beta}{\alpha}  }\right\}, & \text{if~} x\neq 0
						\end{cases}  \\
						& = \begin{cases}
						\b{C}_{\alpha,\beta}, & \text{if~} x =0; \medskip \\
						\displaystyle \left\{ \frac{\rho }{\sqrt{1+\alpha^{2}}} \rb*{\frac{\alpha x}{\norm{x} },1 } \right\}, & \text{if~} x\neq 0.
						\end{cases} 
					\end{align}
				\end{subequations}
			
			To sum up, we have shown that 
			\begin{subequations} 
				\begin{align}
					P_{\bC}\bx 
					 & =  \begin{cases}
					 		\displaystyle \left\{\frac{\rho }{\norm{\bx}}\bx \right\}, & \text{if~} - \xi /\alpha < \norm{x} \leq \alpha \xi; \medskip \\
							\displaystyle \left\{ \frac{\rho }{\sqrt{1+\alpha^{2}}} \rb*{\frac{\alpha x}{\norm{x} },1 } \right\}, & \text{if~} \norm{x} > \max \{ \alpha\xi , -\xi/\alpha \} \text{~or~} \left[\, x\neq 0 \text{~and~} \norm{x} \leq -\xi /\alpha \,\right]; \medskip \\
							\displaystyle \bC_{\alpha,\beta }, & \text{if~} x =0 \text{~and~} 0< -\xi; \medskip \\
							\bC,& \text{if~} \rb{x,\xi } = \rb{0,0} 		
						\end{cases} \\
					& =\begin{cases}
					\displaystyle \left\{\frac{\rho }{\norm{\bx}}\bx \right\}, & \text{if~} \norm{x} \leq \alpha \xi \text{~and~} \xi >0; \medskip  \\
					\displaystyle \left\{ \frac{\rho }{\sqrt{1+\alpha^{2}}} \rb*{\frac{\alpha x}{\norm{x} },1 } \right\}, & \text{if~} \norm{x} > \max \{ \alpha\xi , -\xi/\alpha \} \text{~or~} \left[\, x\neq 0 \text{~and~} \norm{x} \leq -\xi /\alpha \,\right]; \medskip \\
					\displaystyle \sphere{0}{\beta} \times \{\beta/\alpha\}, & \text{if~} x =0 \text{~and~} \xi < 0; \medskip \\
					\bC,& \text{if~} \rb{x,\xi } = \rb{0,0},  					
					\end{cases}
				\end{align}
			\end{subequations}
		as announced in \cref{eq: pr-LC}.
	\end{proof}

\begin{example}\label{eg: Lewis-ver}
	Suppose that $\HH= \bbS^{N}$ is the Hilbert space of symmetric matrices of \cref{sect: SN}. Set $K \coloneqq \bbS_{+}^{N}$, let  $\rho \in \RR_{++}$, and set  $\sC \coloneqq K \cap \sphere{0}{\rho}$. In addition, let $A \in \HH$, and let $U \in \bbU^{N}$ be such that $A = U \rb{ \Diag \lambda\rb{A} } U^{\T}$; set 
		\begin{equation}
			\sD\coloneqq \menge{ V \rb{\Diag\rb{\rho,0,\ldots,0}} V^{\T} }{ V \in \bbU^{N} \text{~such that~} A = V \rb{\Diag \lambda\rb{A} } V^{\T} }
		\end{equation}
	and 
		\begin{equation}
			\sE \coloneqq \sphere{0}{\rho} \cap  \cone\rb{\conv \sD } .
		\end{equation}
	Then 
		\begin{equation}
			 P_{\sC}A 	
				= \begin{cases}
					\displaystyle \left\{ \frac{\rho}{ \norm{ \rb{\lambda\rb{A}}_{+}  } }  U \rb{ \Diag\rb{ \lambda\rb{A}}_{+} } U^{\T} \right\}, & \text{if~} \lambda_{1}\rb{A} >0; \medskip \\
					\sE, & \text{if~} \lambda_{1}\rb{A} =0; \medskip \\
					\sD, & \text{if~} \lambda_{1}\rb{A} < 0.
				\end{cases}
		\end{equation}
\end{example}

	\begin{proof}
		Set 
			\begin{equation}
			\sC_{\rho} \coloneqq \menge*{ B \in \bbS_{+}^{N} }{\rank{B}=1 \text{~and~} \fnorm{B}=\rho }.
			\end{equation}
		It then follows from \cref{lem: SN}\cref{it: pr-Crho} that  
			\begin{equation}\label{eq: max-Crho}
				\max \scal{A}{\sC_{\rho} } = \rho \lambda_{1}\rb{A} \quad \text{and} \quad P_{\sC_{\rho}}A = \sD.
			\end{equation}
		 Let us now  consider all conceivable cases: 
			
			\hspace{\parindent}(A) $\lambda_{1}\rb{A}>0$: Then $\max \scal{A}{\sC_{\rho} } >0$, and thus, by \cref{lem: SN}\cref{it: psd-gen} and \cref{lem: cone}\cref{it: c01}, we obtain $A \in \HH\smallsetminus\pc{K}$. Therefore, since $\{0\} \neq K$ is a nonempty closed convex cone, we infer from \cref{thm: K-S}\cref{it: KS2} and \cref{fact: pr-psd} that 
				\begin{equation}
					P_{\sC}A = \left\{ \frac{\rho}{\fnorm{ P_{K}A } } P_{K}A \right\} = \left\{ \frac{\rho}{ \norm{ \rb{\lambda\rb{A}}_{+}  } }  U \rb{ \Diag\rb{ \lambda\rb{A}}_{+} } U^{\T} \right\}.
 				\end{equation}
 				
 			\hspace{\parindent}(B) $\lambda_{1}\rb{A} \leq 0$: Then $\max \scal{A}{\sC_{\rho} } \leq 0$. Since $\rb{\forall B \in \sC_{\rho}}~\fnorm{B}=\rho$ and, by \cref{lem: SN}\cref{it: psd-gen}, $K = \pos{\sC_{\rho}}$, it follows from \cref{thm: pK}\cref{it: pK22}\&\cref{it: pK23} and \cref{eq: max-Crho} that  
 			\begin{subequations} 
 				\begin{align}
	 				P_{\sC}A 
	 					& = \begin{cases}
		 				P_{\sC_{\rho}}A, & \text{if~} \max \scal{A}{\sC_{\rho} } < 0; \\
		 				\sphere{0}{\rho} \cap \cone\rb{\conv P_{\sC_{\rho}}A }, & \text{if~}\max \scal{A}{\sC_{\rho} } = 0 
	 				\end{cases} \\
	 				& = \begin{cases}
		 				\sD, & \text{if~} \lambda_{1}\rb{A} < 0; \\
		 				\sE, & \text{if~} \lambda_{1}\rb{A} =0,
	 				\end{cases}
 				\end{align}
 			\end{subequations}
 		which completes the proof.
	\end{proof}

\begin{remark}
	Consider the setting of \cref{eg: Lewis-ver}. Since $U \rb{ \Diag\rb{\rho,0,\ldots,0} } U^{\T} \in \sD\subseteq \sE$, we see that 
		\begin{equation}
			\operatorname{s}\colon 
			\HH \to \HH: A \mapsto \begin{cases} 
				\displaystyle \frac{\rho}{ \norm{ \rb{\lambda\rb{A}}_{+}  } }  U \rb{ \Diag\rb{ \lambda\rb{A}}_{+} } U^{\T}, & \text{if~} \lambda_{1}\rb{A} >0 ; \medskip \\
				U \rb{ \Diag\rb{\rho,0,\ldots,0} } U^{\T}, &\text{otherwise}
			\end{cases}
		\end{equation} 
	is a selection of $P_{\sC}$.
\end{remark}

\section{Copositive matrices: a numerical experiment}

\label{s:copos}

In this final section, 
$N$ is a strictly positive integer and $M$ is a symmetric matrix
in $\RR^{N\times N}$.  Recall that $M$ is \emph{copositive} if 
	\begin{math}
		\rb{\forall x \in \RR^{N}_{+} } ~  \scal{x}{Mx} \geq 0 
	\end{math};
or, equivalently, 
	\begin{equation}\label{eq: copo}
		\mu\rb{M} \coloneqq  \min_{x \in \RR_{+}^{N}\cap
		\sphere{0}{1}} \tfrac{1}{2} \scal{x}{Mx} \geq 0.
	\end{equation} 
For further information on copositive matrices, we refer the reader to the surveys \cite{dur2010copositive, hiriart2010variational} and references therein. In view of \cref{eq: copo}, testing copositivity of $M$ amounts to 	
	\begin{equation}\label{eq: copo2}
	\minimize{x \in \RR_{+}^{N}\cap \sphere{0}{1} }{\tfrac{1}{2} \scal{x}{Mx} }.
	\end{equation}
Now, set $C \coloneqq \RR_{+}^{N} \cap \sphere{0}{1}$, set $f
\colon \RR^{N} \to \RR : x \mapsto \rb{1/2}\scal{x}{Mx}$, and set
$g \coloneqq \iota_{C}$ which is the indicator function of $C$.
Note that neither $f$ nor $g$ is convex;
 however, $\nabla f$ is Lipschitz continuous
with the operator norm $\|M\|$ (computed as the largest singular
value of $M$) being a suitable Lipschitz
constant.
The projection onto $C$ is computed using \cref{eq: eg-RN}. 
In turn, \cref{eq: copo2} can be written as  
	\begin{equation}\label{eq: copo3}
		\minimize{x \in \RR^{N} }{f\rb{x} + g\rb{x} }.
	\end{equation}
To solve this problem, we compared the
\emph{Fast Iterative Shrinkage-Thresholding Algorithm (FISTA)} 
(see \cite{beck2009fast}), 
the \emph{Projected Gradient Method (PGM)}  
(see \cite{attouch2013convergence, bolte2014proximal}), 
the algorithm presented in \cite[Example 5.5.2]{lange2016mm} by
Lange, 
the \emph{Douglas{\textendash}Rachford Algorithm (DRA)} variant 
presented in \cite{LiPongDR} by Li and Pong, 
and the regular DRA for solving \cref{eq: copo3} when $N \in \left\{2,3,4\right\}$. 
For each $N \in \left\{2,3,4\right\}$, using the  copositivity
criteria for matrices of order up to four (see, e.g.,
\cite{hadeler1983copositive, ping1993criteria}),  we randomly
generate 100 copositive matrices (group A) together with 100
non-copositive (group B) ones.
For each algorithm,   if $\fa{x_{n}}{n \in \NN}$ is the sequence
generated, then we terminate the algorithm when 
	 \begin{equation}
		 \frac{\norm{x_{n}- x_{n-1}} }{\max\{ \norm{x_{n-1}},1 \}} < 10^{-8}.
	 \end{equation}
The maximum allowable number of iterations is $1000$. 
For each matrix $M$ in group A (respectively, group $B$), we declare success if $\mu\rb{M} \geq 0$ (respectively, $\mu\rb{M} < 0$). We also record the average of the number of iterations until success of each algorithm. 
The results, obtained using \texttt{Matlab}, are reported in \cref{tab}.
 
 \begin{table}[H]
 	\centering
 	\begin{tabular}{*{12}c}
 		\toprule[1 pt]
 		\multirow{2}{*}{Size}        &
		\multirow{2}{*}{Copositive} &
		\multicolumn{2}{c}{FISTA} &
		\multicolumn{2}{c}{PGM} & \multicolumn{2}{c}{Lange} & \multicolumn{2}{c}{Li{\textendash}Pong} & \multicolumn{2}{c}{DR} \\ \cmidrule[1pt](lr){3-4}
 		\cmidrule[1pt](lr){5-6} \cmidrule[1pt](lr){7-8} \cmidrule[1pt](lr){9-10} \cmidrule[1pt](lr){11-12} \addlinespace[.5 ex] 
 														         & &succ    &  avg iter        &succ       & avg iter    & succ      & avg iter    & succ        & avg iter  & succ     & avg iter      \\ \midrule[1 pt]
 		\multirow{2}{*}{$2\times 2$} & Yes 					       &100     &  5               &100        & 5           & 100       & 89          & 100         & 94        & 96       & 23            \\ 
 									 & No                          &97      &  15              & 99        & 12          & 91        & 92          & 93          & 87        & 53       & 89            \\ \midrule[1 pt]
 		\multirow{2}{*}{$3\times 3$} & Yes 					       &100     &  27              &100        & 24          & 100       & 91          & 100         & 232       & 95       & 63            \\ 
 									 & No                          &96      &  30              & 98        & 24          & 86        & 93          & 95          & 162       & 31       & 214           \\ \midrule[1 pt] 
 		\multirow{2}{*}{$4\times 4$} & Yes 						   &100     &  60              &100        & 62          & 100       & 90          & 100         & 482       & 85       &126	        \\ 
 									 & No						   &100		& 51               &100        & 45          & 94        & 95          & 100         & 264       & 11       &114            \\ \bottomrule[1 pt]
 	\end{tabular} 
 \caption{Detecting whether a matrix is copositive using a
 variety of algorithms.}\label{tab}
 \end{table}
 
Finally, let us apply the algorithms to the well-known Horn
 matrix 
 	\begin{equation}
	 	H \coloneqq \begin{bmatrix*}[r]
		 	1 & -1 & 1 & 1 & -1 \\ 
		 	-1 & 1 & -1 & 1 &1 \\
		 	1 & -1 & 1 & -1 & 1 \\
		 	1 & 1 & -1 & 1 &-1 \\
		 	-1 & 1 & 1 & -1 & 1
	 	\end{bmatrix*},
 	\end{equation}
 which is copositive with $\mu(H)=0$ (see \cite[Equation~(3.5)]{hall1963copositive}). 
For each algorithm, we record the number of iterations and the
value of $f$ at the point that the algorithm is terminated. The
results are recorded in \cref{tabtwo}. 
	
	\begin{table}[H]
		\centering
		\begin{tabular}{*{10}c}
			\toprule[1 pt]
			\multicolumn{2}{c}{FISTA} &
			\multicolumn{2}{c}{PGM}        & \multicolumn{2}{c}{Lange}       & \multicolumn{2}{c}{Li{\textendash}Pong}      & \multicolumn{2}{c}{DR} \\ \cmidrule[1pt](lr){1-2}
			\cmidrule[1pt](lr){3-4} \cmidrule[1pt](lr){5-6} \cmidrule[1pt](lr){7-8} \cmidrule[1pt](lr){9-10} \addlinespace[.5 ex] 
			fval   &   iter           &  fval          &   iter        &fval   &   iter                &fval   &   iter              &fval   &   iter   \\ \midrule[1 pt] 
			$3.5230\mathrm{e}{-}17$    & 11              & $2.8297\mathrm{e}{-}20$&10       &$2.9979\mathrm{e}{-}07$&95       &$1.4912\mathrm{e}{-}14$& 170   & $ 0.0584$ & 13
		\end{tabular}
\caption{Detecting copositivity of the Horn matrix.}
		\label{tabtwo}
	\end{table} 

We acknowledge that these algorithms might get stuck at points
that are not solutions
and that the outcome might depend on the starting points;
moreover, a detailed complexity analysis is absent. 
There are thus various research opportunities to 
improve the current results. Nonetheless, our preliminary results indicate that FISTA and PGM
are potentially significant contenders for numerically testing copositivity.

\paragraph{Acknowledgments}{   
The authors thank the referees, 
Amir Beck, Minh Dao, Kenneth Lange, Marc
Teboulle, and Henry Wolkowicz for helpful comments and for referring us
to additional references.
HHB and XW were partially supported by NSERC Discovery Grants;
MNB was partially supported by a Mitacs Globalink Graduate Fellowship Award.
}


\bibliographystyle{abbrv}

\bibliography{reference}

\begin{thebibliography}{10}

\bibitem{anderson1969series}
W.~N. Anderson and R.~J. Duffin.
\newblock Series and parallel addition of matrices.
\newblock {\em Journal of Mathematical Analysis and Applications},
  26(3):576--594, 1969.

\bibitem{attouch2013convergence}
H.~Attouch, J.~Bolte, and B.~F. Svaiter.
\newblock {Convergence of descent methods for semi-algebraic and tame problems:
  proximal algorithms, forward--backward splitting, and regularized
  Gauss--Seidel methods}.
\newblock {\em Mathematical Programming (Series A)}, 137(1-2):91--129, 2013.

\bibitem{bauschke2017convex}
H.~H. Bauschke and P.~L. Combettes.
\newblock {\em Convex {A}nalysis and {M}onotone {O}perator {T}heory in Hilbert
  {S}paces}.
\newblock Springer International Publishing, second edition, 2017.

\bibitem{beck2009fast}
A.~Beck and M.~Teboulle.
\newblock A fast iterative shrinkage-thresholding algorithm for linear inverse
  problems.
\newblock {\em SIAM Journal on Imaging Sciences}, 2(1):183--202, 2009.

\bibitem{bolte2014proximal}
J.~Bolte, S.~Sabach, and M.~Teboulle.
\newblock Proximal alternating linearized minimization for nonconvex and
  nonsmooth problems.
\newblock {\em Mathematical Programming (Series A)}, 146(1-2):459--494, 2014.

\bibitem{dur2010copositive}
M.~D{\"u}r.
\newblock Copositive programming -- a survey.
\newblock In M.~Diehl, F.~Glineur, E.~Jarlebring, and W.~Michiels, editors,
  {\em Recent Advances in Optimization and its Applications in Engineering}.
  Springer, 2010.

\bibitem{ferreira2013projections}
O.~P. Ferreira, A.~N. Iusem, and S.~Z. N{\'e}meth.
\newblock Projections onto convex sets on the sphere.
\newblock {\em Journal of Global Optimization}, 57(3):663--676, 2013.

\bibitem{hadeler1983copositive}
K.~Hadeler.
\newblock On copositive matrices.
\newblock {\em Linear Algebra and its Applications}, 49:79--89, 1983.

\bibitem{hall1963copositive}
M.~Hall and M.~Newman.
\newblock Copositive and completely positive quadratic forms.
\newblock {\em Mathematical Proceedings of the Cambridge Philosophical
  Society}, 59(2):329--339, 1963.

\bibitem{haugazeau1968inequations}
Y.~Haugazeau.
\newblock Sur les in{\'e}quations variationnelles et la minimisation de
  fonctionnelles convexes.
\newblock {\em These, Universit\'e de Paris}, 1968.

\bibitem{hiriart2010variational}
J.-B. Hiriart-Urruty and A.~Seeger.
\newblock A variational approach to copositive matrices.
\newblock {\em SIAM Review}, 52(4):593--629, 2010.

\bibitem{lange2016mm}
K.~Lange.
\newblock {\em MM {O}ptimization {A}lgorithms}.
\newblock SIAM, 2016.

\bibitem{lewis2008alternating}
A.~S. Lewis and J.~Malick.
\newblock Alternating projections on manifolds.
\newblock {\em Mathematics of Operations Research}, 33(1):216--234, 2008.

\bibitem{LiPongDR}
G.~Li and T.~K. Pong.
\newblock Douglas-{R}achford splitting for nonconvex optimization with
  application to nonconvex feasibility problems.
\newblock {\em Mathematical Programming (Series A)}, 159(1-2):371--401, 2016.

\bibitem{moreau1962decomposition}
J.-J. Moreau.
\newblock D{\'e}composition orthogonale d'un espace hilbertien selon deux
  c{\^o}nes mutuellement polaires.
\newblock {\em Comptes Rendus Hebdomadaires des S\'eances de l'Acad\'emie des
  Sciences (S\'eries A)}, 225:238--240, 1962.

\bibitem{ping1993criteria}
L.~Ping and F.~Y. Yu.
\newblock Criteria for copositive matrices of order four.
\newblock {\em Linear Algebra and its Applications}, 194:109--124, 1993.

\bibitem{rocky}
R.~T. Rockafellar.
\newblock {\em Convex {A}nalysis}.
\newblock Princeton University Press, 1970.

\bibitem{theobald1975inequality}
C.~M. Theobald.
\newblock {A}n {i}nequality for the {t}race of the {p}roduct of {t}wo
  {s}ymmetric {m}atrices.
\newblock {\em Mathematical Proceedings of the Cambridge Philosophical
  Society}, 77:265--267, 1975.

\bibitem{yu2013decomposing}
Y.-L. Yu.
\newblock On decomposing the proximal map.
\newblock In C.~J.~C. Burges, L.~Bottou, M.~Welling, Z.~Ghahramani, and
  K.~Weinberger, editors, {\em Advances in Neural Information Processing
  Systems 26}, pages 91--99. Curran Associates, Inc., 2013.

\end{thebibliography}

\end{document}